\documentclass[12pt,british,a4paper,reqno]{amsart}
\usepackage[T1]{fontenc}
	\usepackage[latin9]{inputenc}
	\usepackage[british]{babel}

	\usepackage{amsmath}
	\usepackage{amssymb}
	\usepackage{amsthm}
	\usepackage{braket}
	\usepackage{xcolor}
\usepackage[shortlabels]{enumitem}
	\usepackage{fancyhdr}
	\usepackage{geometry}
	\usepackage{graphicx}
	\usepackage{hyperref}
		\usepackage{bookmark}
	\usepackage{ifsym}
	\usepackage{indentfirst}
	\usepackage{mathrsfs}
	\usepackage{mathtools}
	\usepackage{proof}
	\usepackage{qtree}
	\usepackage{setspace}
	\usepackage{stmaryrd}
	\usepackage{tikz}
	\usepackage{tikz-cd}
	\usepackage{tabstackengine}
	\setstackgap{L}{.75\baselineskip}
\usetikzlibrary{chains}
\usepackage{cleveref}
	\bookmarksetup{numbered,open}

	\providecommand{\corollaryname}{Corollary}
	\providecommand{\definitionname}{Definition}
	\providecommand{\examplename}{Example}
	\providecommand{\lemmaname}{Lemma}
	\providecommand{\propositionname}{Proposition}

	\providecommand{\remarkname}{Remark}

	\providecommand{\theoremname}{Theorem}

	\theoremstyle{plain}
		\newtheorem{thm}{\protect\theoremname}[section] 
		\newtheorem{prop}[thm]{\protect\propositionname}
		\newtheorem{lem}[thm]{\protect\lemmaname}
		\newtheorem{cor}[thm]{\protect\corollaryname}

	\theoremstyle{definition}
		\newtheorem{defn}[thm]{\protect\definitionname}

	\theoremstyle{remark}

	\numberwithin{figure}{section}
	\numberwithin{equation}{section}

	\newenvironment{acknowledgements}{
		\begin{abstract}} {\end{abstract}}

     \newcommand{\fk}[1]{\mathbf{#1}}

		\newcommand{\Hom}{\operatorname{Hom}\nolimits}

            \newcommand{\Func}{\operatorname{Func}\nolimits}

		\newcommand{\rMod}[1]{\operatorname{\mathbf{Mod}\hspace{0.3mm}\text{--}\hspace{-0.7mm}}\nolimits{#1}}
		\newcommand{\lMod}[1]{{#1}\operatorname{\hspace{-0.7mm}\text{--}\hspace{0.3mm}\mathbf{Mod}}\nolimits}
		
 \newcommand{\myid}{\mathbf{1}}



\title{Corner replacement for Morita contexts}
\author{Raphael Bennett-Tennenhaus}
\date{}

\begin{document}
\maketitle

\begin{abstract}
We consider how Morita equivalences are compatible with the notion of a corner subring. 
Namely, we outline a canonical way to replace a corner subring of a given ring with one which is Morita equivalent, and look at how such an equivalence ascends. 

We use the language of Morita contexts, and then specify these more general results. 
We give applications to trivial extensions of finite-dimensional algebras, tensor rings of pro-species,  semilinear clannish algebras arising from orbifolds, and functor categories. 
\end{abstract}

\section{Introduction}

Celebrated results of Morita from \cite[\S3]{Morita-duality-modules-rings-with-min-condtions} characterise the existence of an equivalence $\lMod{R}\simeq \lMod{S}$ between module categories in terms of the existence of a pair ${}_{R}N_{S},{}_{S}L_{R}$ of bimodules subject to compatibility conditions.  
These conditions include the existence of bimodule homomorphisms of the form $\theta\colon N\otimes L\to R$ and  $\zeta\colon L\otimes N\to S$ satisfying certain associativity conditions. 
A \emph{Morita context}, also known as \emph{pre}-\emph{equivalence data}, refers to the tuple $(R,S;\,N,L;\,\theta,\zeta)$. 
We note, in particular, that such a context defines an equivalence of the form above when the maps $\theta,\zeta$ are surjective; see for example \Cref{thm-recalled-from-Anh-Marki}. 
The results of Morita are presented in an illustrious book by Bass \cite[Chapters II, III]{Bass-algebraic-K-theory}.

Morita contexts have since been studied in various other generalities. In particular, so-called \emph{Morita theory} has been developed by  \'{A}nh and M\'{a}rki \cite{Marki-Anh-Morita-local-units-1987} to unital modules over rings with local units, and more generally, to categories of modules over idempotent rings in work of Garc\'{\i}a and Sim\'{o}n \cite{Garcia-Simon-Morita-equivalence-idempotent-rings}. An idea we exploit here is that any Morita context as above gives rise to a \emph{ring of Morita context}, a ring of matrices denoted $\begin{bsmallmatrix}
R & N
\\
L & S
\end{bsmallmatrix}$ whose ring structure is canonically defined by information from the context. 

The general purpose of this note is to record the statement and a proof of Proposition \ref{prop-induced-contexts-for-calssical-matrix-rings} below. 
The author suspects Proposition \ref{prop-induced-contexts-for-calssical-matrix-rings} was already well-known `folklore'. 

\begin{prop}
\label{prop-induced-contexts-for-calssical-matrix-rings}
Let $R$ and $S$ be unital associative rings and let $e\in R$ be an idempotent. 

If ${}_{eRe}M_{S}$ and ${}_{S}N_{eRe}$ are bimodules defining a Morita equivalence $eRe\sim S$, then 
\[
\begin{array}{cc}
\begin{pmatrix}
(1-e)R(1-e) & (1-e)Re\otimes M
\\
eR(1-e) & M
\end{pmatrix},
&
\begin{pmatrix}
(1-e)R(1-e) & (1-e)Re
\\
N\otimes eR(1-e) & N
\end{pmatrix}
\end{array}
\]
are bimodules defining a Morita equivalence $R\sim    \begin{bsmallmatrix}
(1-e)R(1-e) & (1-e)Re\otimes M
\\
N\otimes eR(1-e) & S
\end{bsmallmatrix}$. 
\end{prop}

Note that the rings $R$ and $S$ are unital, however, in \Cref{thm-induced-contexts-for-generalised-matrix-rings}, we work in the setting of Morita contexts for (possibly non-unital) rings with local units, and unital bimodules over them. 
Working in this more general setting extends the reach of potential applications. 
\section{Background on Morita contexts}

\subsection{Rings with local units} 
\cite[Definition 1]{Marki-Anh-Morita-local-units-1987} 
By a \emph{ring} we will mean one which is associative, but non-unital, meaning we do not require it to have a multiplicative identity. 

If $R$ is a ring we write $\mu_{R}\colon R\otimes_{R} R\to R$ for the universal $R$-$R$-bimodule homomorphism given by the (balanced) multiplication map $R\times R\to R$. 
If $M$ is a left module over $R$ we likewise write $\lambda^{M}_{R}\colon R\otimes_{R}M\to M$ for the universal left $R$-module homomorphism given by the left action map $R\times M\to M$. 
Similarly define $\rho^{S}_{M}\colon M\otimes_{S}S\to M$ if $M$ is a right $S$-module for some ring $S$, meaning that $M$ is an $R\text{-}S$ bimodule if and only if one has $\lambda^{R}_{M} (\myid_{R}\otimes \rho^{S}_{M})=\rho^{S}_{M} (\lambda_{M}^{R}\otimes \myid_{S})$ for a pair of such actions. 

A ring $R$ is said to  \emph{have local units} if every finite subset of $R$ is contained in a (unital) subring of the form $eRe$ for some $e=e^{2}\in R$. 
For a left module $M$ over a  ring $R$ with local units, we say $M$ is \emph{unital} if $\lambda_{M}^{R}$ is surjective. 
Similarly if $S$ has local units then we call a right $S$-module $M$ unital if $\rho_{M}^{S}$ is surjective, and we call a bimodule unital if it is both unital as a left module and as a right module.

\subsection{Matrix notation}

In what remains in the article we use the following notation. 

Let $(X_{ij}\mid m,n)$ be a  collection of additive groups $X_{ij}$ for each $i=1,\dots,m$ and each $j=1,\dots,n$. 
Throughout we consider the following group under component-wise addition,
\[
\begin{array}{c}
\begin{pmatrix}
X_{11} & \cdots & X_{1m} \\
\vdots & \ddots & \vdots \\
X_{l1} & \cdots & X_{lm}
\end{pmatrix}
=
\left\{
\begin{pmatrix}
x_{11} & \cdots & x_{1m} \\
\vdots & \ddots & \vdots \\
x_{l1} & \cdots & x_{lm}
\end{pmatrix}
\Bigg\vert\,
x_{ij}\in X_{ij} \text{ for each }i\text{ and }j
\right\},
\end{array}
\]
which is isomorphic to the direct sum of the $X_{ij}$. 

Consider now collections $(X_{ik}\mid l,m)$, $(Y_{kj}\mid m,n)$ and $(Z_{ij}\mid l,n)$ of additive groups. 
Suppose additionally there are rings $A_{1},\dots,A_{l}$ such that each $X_{ik}$ is a right $A_{k}$-module
 and each $Y_{kj}$ is a left $A_{k}$-module. 
 Furthermore, suppose there is a collection of group homomorphisms $\alpha_{ikj}\colon X_{ik} \otimes_{A_{k}} Y_{kj} \to Z_{ij}$ uniquely defined by $A_{k}$-balanced maps of the form $ X_{ik} \times Y_{kj} \to Z_{ij}$. 
Combining the functions $\alpha_{ikj}$ we use the additive structure of the groups involved to consider a biadditive map $\left\vert
\alpha_{ikj}
\right\vert$,  denoted in detail by
\[
\begin{array}{c}
\begin{vmatrix}
 \alpha_{1k1} & \cdots & \alpha_{1km} \\
\vdots & \ddots & \vdots \\
\alpha_{lk1} & \cdots & \alpha_{lkm}
\end{vmatrix}
\colon
\begin{pmatrix}
X_{11} & \cdots & X_{1m} \\
\vdots & \ddots & \vdots \\
X_{l1} & \cdots & X_{lm}
\end{pmatrix}
\times 
\begin{pmatrix}
Y_{11} & \cdots & Y_{1n} \\
\vdots & \ddots & \vdots \\
Y_{m1} & \cdots & Y_{mn}
\end{pmatrix}
\to 
\begin{pmatrix}
Z_{11} & \cdots & Z_{1n} \\
\vdots & \ddots & \vdots \\
Z_{l1} & \cdots & Z_{ln}
\end{pmatrix}
\end{array}
\]
which is defined by the assignment $((x_{ik}),(y_{kj}))\mapsto (z_{ij})$ where $z_{ij}=\sum_{k}\alpha_{ikj}(x_{ik}\otimes y_{kj})$.

In this notation we can consider some familiar situations.  
\begin{itemize}[--]
    \item In \S\ref{subsec-morita-context-ring}  we take $l=m=n$ and $X_{ij}=Y_{ij}=Z_{ij}$ where each $X_{ii}$ is a ring so that  $X_{ij}$ is an $X_{ii}$-$X_{jj}$-bimodule, such that the map $\left\vert
    \alpha_{ikj}
    \right\vert$ defines a ring multiplication.
    \item In \S\ref{subsec-morita-context-modules} we take $l=m$, $n=1$,  $Y_{ik}=Z_{ik}$ so that $\left\vert
    \alpha_{ikj}
    \right\vert$ defines a left module action over a ring from the item above. 
    Dually one can define right modules. 
    \item In \S\ref{subsec-morita-context-homomorphisms-of-modules}, \S\ref{subsec-morita-context-balanced-maps-of-modules}, \S\ref{subsec-morita-context-bimodules} and \S\ref{subsec-morita-context-associativty-of-balanced-maps} we similarly use the notation above to  consider module homomorphisms, balanced maps between modules, bimodules and associativiy conditions, respectively. 
    This sets up notation for the remainder of the article.
\end{itemize}

\subsection{Matrix ring of a generalised Morita context}
\label{subsec-morita-context-ring} 
By a \emph{generalised Morita context} we will mean a triple $(A_{i};\,M_{ij};\,\varphi_{ikj})$ where the following conditions hold.
\begin{itemize}
    \item $A_{i}$ ($i=1,\dots,n$) are associative rings. 
    \item $M_{ij}$ ($i,j=1,\dots,n$) are  $A_{i}$-$A_{j}$-bimodules, where $M_{ii}=A_{i}$. 
    \item $\varphi_{ikj}\colon M_{ik}\otimes_{A_{k}} M_{kj}\to M_{ij}$ ($i,j,k=1,\dots,n$) are $A_{i}$-$A_{j}$-bimodule maps, where
    \[
    \begin{array}{ccc}
\varphi_{kkj}=\lambda^{A_{k}}_{M_{kj}},
 &
 \varphi_{kkk}=\mu_{A_{k}},
 &
 \varphi_{ikk}=\rho^{A_{k}}_{M_{ik}}.
    
    \end{array}
    \]
    \item $\varphi_{ihj}(\myid\otimes \varphi_{hkj})=\varphi_{ikj}(\varphi_{ihk}\otimes\myid)$ ($h,i,j,k=1,\dots, n$) as maps $M_{ih}\otimes M_{hk}\otimes M_{kj}\to M_{ij}$.
\end{itemize}
The equations in the last item above are referred to as the \emph{mixed associativity conditions}. 
Consequently, as in \S\ref{subsec-morita-context-ring}, with respect to $(A_{i};\,M_{ij};\,\varphi_{ikj})$ there is a \emph{generalised matrix ring}
\[
[A_{i}, M_{ij}]
=
\begin{bmatrix}
A_{1} & \cdots & M_{1n} \\
\vdots & \ddots & \vdots \\
M_{n1} & \cdots & A_{n}
\end{bmatrix}
\text{ with multiplication map }
\left\vert\varphi_{ikj}\right\vert=
\begin{vmatrix}
 \varphi_{1k1} & \cdots & \varphi_{1kn} \\
\vdots & \ddots & \vdots \\
\varphi_{nk1} & \cdots & \varphi_{nkn}
\end{vmatrix}.
\] 
Note that if each ring $A_{i}$ has local units and each bimodule $M_{ij}$ is unital, then it is straightforward to see that $[A_{i}, M_{ij}]$ has local units. 
Morita contexts are usually considered as above, but for $n=2$; see for example \cite[p. 41]{Garcia-Simon-Morita-equivalence-idempotent-rings}. 
The following result of \'{A}nh and M\'{a}rki \cite{Marki-Anh-Morita-local-units-1987} is part of a Morita theory  developed for rings with local units.
\begin{thm}
\label{thm-recalled-from-Anh-Marki}
\cite[Theorem 2.2]{Marki-Anh-Morita-local-units-1987} Let $(R,S;\,N,L;\,\theta,\zeta)$ be a Morita context 
 where $R$ and $S$ have local units and ${}_{R}N_{S}$ and ${}_{S}L_{R}$ are unital, 
 and let $\lMod{R}$ and $\lMod{S}$ be the categories of unital left modules over $R$ and $S$ respectively.

Then $N\otimes_{S}-\colon \lMod{S}\to \lMod{R}$ and $L\otimes_{R}-\colon \lMod{R}\to \lMod{S}$ are mutually quasi-inverse equivalences if and only if the bimodule homomorphisms $\theta $ and $\zeta$ are surjective. 
\end{thm}
\subsection{Modules over matrix rings}
\label{subsec-morita-context-modules}
Let 
$(A_{i};\,M_{ij};\,\varphi_{ikj})$ be a  generalised Morita context in the sense of \S\ref{subsec-morita-context-ring}. 
By a \emph{column collection for} $(A_{i};\,M_{ij};\,\varphi_{ikj})$ we mean an $n$-tuple $P_{\fk{c}}=(P_{1},\dots,P_{n})$ of  left $A_{i}$-modules $P_{i}$ ($i=1,\dots,n)$ where the following conditions hold. 
    \begin{itemize}
        \item $\beta_{ik} \colon M_{ik}\otimes_{A_{k}} P_{k}\to P_{i}$ ($i,k=1,\dots,n$) are left $A_{i}$-module homomorphisms, where $\beta_{ii}=\lambda_{P_{i}}^{A_{i}}$ for each $i$, which define the biadditive map $\left\vert\beta_{ik}\right\vert$, denoted in detail by
        \[
\begin{array}{c}
\begin{vmatrix}
 \beta_{1k}  \\
\vdots \\
\beta_{nk}
\end{vmatrix}
\colon
\begin{bmatrix}
A_{1} & \cdots & M_{1n} \\
\vdots & \ddots & \vdots \\
M_{n1} & \cdots & A_{n}
\end{bmatrix}
\times 
\begin{pmatrix}
P_{1}\\
\vdots  \\
P_{n} 
\end{pmatrix}
\to 
\begin{pmatrix}
P_{1}\\
\vdots  \\
P_{n} 
\end{pmatrix}
\end{array}
\]
        \item $\beta_{ih}(\varphi_{ikh}\otimes \myid)=\beta_{ik}(\myid\otimes \beta_{kh})$ ($h=1,\dots, n$) as maps $M_{ik}\otimes M_{kh}\otimes P_{h}\to P_{i}$.
    \end{itemize}
    These conditions imply that the modules $(P_{1},\dots, P_{n})$ together with the (left) \emph{structure maps} $\beta_{ik}$ define an object in the category of (left) \emph{representations} of the given generalised Morita context; see \cite[\S 2, (2.1)]{Chen-Hongjin-derived-equivalences-2020}. 
    In other words, $P_{\fk{c}}$ together with the map $\left\vert\beta_{ik}\right\vert$ defines a left module over the generalised matrix ring $[A_{i}, M_{ij}]$. 
    
    By a result of Chen and Liu \cite[Proposition 2.1]{Chen-Hongjin-derived-equivalences-2020}, which generalises a result  of Green \cite[Theorem 1.5]{Green-rings-in-matrix-form}, if $A_{1},\dots,A_{n}$ are unital then the correspondence above defines a categorical  equivalence. 
    In \cite[\S 2]{Chen-Hongjin-derived-equivalences-2020} these authors consider right $[A_{i}, M_{ij}]$-modules, constructed dually to the column collection $P_{\fk{c}}$, which we refer to as a \emph{row collection}.

\subsection{Module homomorphisms over matrix rings}
\label{subsec-morita-context-homomorphisms-of-modules} 
Let $(A_{i};\,M_{ij};\,\varphi_{ikj})$ be a generalised Morita context in the sense of \S\ref{subsec-morita-context-ring}. 
As in \S\ref{subsec-morita-context-modules}, let $P_{\fk{c}}=(P_{1},\dots,P_{n})$ and $P_{\fk{c}}'=(P_{1}',\dots,P_{n}')$ be column collections for $(A_{i};\,M_{ij};\,\varphi_{ikj})$, corresponding to left $[A_{i};\,M_{ij}]$-modules via left structure maps  $\beta_{ik} \colon M_{ik}\otimes_{A_{k}} P_{k}\to P_{i}$ and $\beta_{ik}' \colon M_{ik}\otimes_{A_{k}} P_{k}'\to P_{i}'$  respectively. 
By a \emph{column morphism} we mean a $n$-tuple $f_{\fk{c}}= (f_{1},\dots,f_{n})$ where the following conditions hold.
\begin{itemize}
    \item $f_{i}\colon N_{i}\to N_{i}'$ ($i=1,\dots, n$) are left $A_{i}$-module homomorphisms.
    \item $f_{i}\beta_{ij}=\beta_{ij}'(\myid\otimes f_{j})$ ($i,j=1,\dots, n$) as maps $M_{ij}\otimes P_{j}\to P_{i}'$.
\end{itemize}
Then, as in  \S\ref{subsec-morita-context-modules}, by following dual arguments in \cite[\S 2]{Chen-Hongjin-derived-equivalences-2020} the tuple $(f_{1},\dots,f_{n})$ defines an $[A_{i};\,M_{ij}]$-module homomorphism $f_{\fk{c}}\colon P_{\fk{c}}\to P_{\fk{c}}'$ by $f_{\fk{c}}(p_{1},\dots,p_{n})=(f_{1}(p_{1}),\dots,f_{n}(p_{n}))$.  

Following \cite[\S 2]{Chen-Hongjin-derived-equivalences-2020} directly gives a dual construction of a right $[A_{i};\,M_{ij}]$-module homomorphism between row collections, which we refer to as a \emph{row morphism}. 
\subsection{Balanced homomorphisms over matrix rings}
\label{subsec-morita-context-balanced-maps-of-modules}
As in \S\ref{subsec-morita-context-ring}--\ref{subsec-morita-context-modules}, for  a generalised Morita context $(A_{i};\,M_{ij};\,\varphi_{ikj})$, fix a row collection $P_{\fk{r}}=(P_{1},\dots,P_{n})$ with respect to right structure maps $\beta_{ik}\colon P_{i}\otimes M_{ik}\to P_{k}$ and fix a column collection $Q_{\fk{c}}=(Q_{1},\dots,Q_{n})$  with respect to left 
 structure maps $\gamma_{kj}\colon M_{kj}\otimes Q_{j}\to Q_{k}$,  where the following conditions hold.  
 \begin{itemize}
     \item $Z$ is an additive group.
     \item $\alpha_{k}\colon P_{k} \otimes_{A_{k}} Q_{k} \to Z$ ($k=1,\dots,n$) are additive maps defining a biadditive map 
     \[
     \begin{array}{cc}
       \left\vert
\alpha_{k}
\right\vert\colon 
\begin{pmatrix}
P_{1} 
&
\cdots 
&
P_{n} 
\end{pmatrix}
\times 
\begin{pmatrix}
Q_{1}
\\
\vdots  
\\
Q_{n} 
\end{pmatrix}
\to Z   &  ((p_{1},\dots,p_{n}),(q_{1},\dots,q_{n}))\mapsto \sum_{k} \alpha_{k}(p_{k}\otimes q_{k})
     \end{array}
     \]
\item $\alpha_{j}(\myid\otimes \gamma_{ij})
=\alpha_{i}(\beta_{ij}\otimes \myid)$ ($i,j=1,\dots,n$) as maps $P_{i}\otimes M_{ij}\otimes Q_{j}\to Z$. 
 \end{itemize}
Then $\left\vert
\alpha_{k}
\right\vert$ is  $[A_{i};\,M_{ij}]$-balanced, so factors by an additive map $\left\vert
\underline{\alpha}_{k}
\right\vert\colon P_{\fk{r}}\otimes_{[A_{i};\,M_{ij}]}Q_{\fk{c}}\to Z$.

\subsection{Matrix bimodules over matrix rings}
\label{subsec-morita-context-bimodules}
Let $(A_{i};\,M_{ij};\,\varphi_{ikj}),(A_{i}';\,M_{ij}';\,\varphi_{ikj}')$ be  generalised Morita contexts.
Fix additive groups $T_{ij}$ where the following holds. 
\begin{itemize}
    \item $T_{\fk{c}j}=(T_{1j},\dots , T_{nj})$ ($j=1,\dots,n$) are column collections for $(A_{i};\,M_{ij};\,\varphi_{ikj})$  with respect to left 
 structure maps $\gamma_{ikj}\colon M_{ik}\otimes T_{kj}\to T_{ij}$, summarised by 
    \[
\begin{array}{cc}
\begin{vmatrix}
 \gamma_{1kj}  \\
\vdots \\
\gamma_{nkj} 
\end{vmatrix}
\colon
\begin{bmatrix}
A_{1} & \cdots & M_{1n} \\
\vdots & \ddots & \vdots \\
M_{n1} & \cdots & A_{n}
\end{bmatrix}
\times 
\begin{pmatrix}
T_{1j}\\
\vdots  \\
T_{nj} 
\end{pmatrix}
\to 
\begin{pmatrix}
T_{1j}\\
\vdots  \\
T_{nj} 
\end{pmatrix}
, 
\\
\left((m_{ik}),(t_{kj})\right)
\mapsto 
\left(
\sum_{k}\gamma_{ikj}(m_{ik}\otimes t_{kj})
\right).
\end{array}
\]
 \item $T_{i\fk{r}}=(T_{i1},\dots , T_{in})$ ($i=1,\dots,n$) are row collections for $(A_{i}';\,M_{ij}';\,\varphi_{ikj}')$ with respect to right structure maps $\beta_{ikj}\colon T_{ik}\otimes M_{kj}'\to T_{ij}$, summarised by
 \[
\begin{array}{cc}
\begin{vmatrix}
 \beta_{ik1} & \cdots & \beta_{ikn} 
\end{vmatrix}
\colon
\begin{pmatrix}
T_{i1} & \cdots & T_{in}
\end{pmatrix}
\times 
\begin{bmatrix}
A_{1}' & \cdots & M_{1n}' \\
\vdots & \ddots & \vdots \\
M_{n1}' & \cdots & A_{n}'
\end{bmatrix}
\to 
\begin{pmatrix}
T_{i1} & \cdots & T_{in}
\end{pmatrix},
\\
((t_{ik}),(m_{kj}'))
\mapsto 
\left(
\sum_{k}\beta_{ikj}(t_{ik}\otimes m_{kj}')
\right).
\end{array}
\]
\item $\beta_{ikj}(\gamma_{ihk}\otimes \myid )=\gamma_{ihj}(\myid \otimes \beta_{hkj})$ ($h,i,j,k=1,\dots,n$) as maps $M_{ih}\otimes T_{hk} \otimes M_{kj}'\to T_{ij}$.
\end{itemize}
These conditions imply that the collection $(T_{ij})$ defines an  $[A_{i};\,M_{ij}]$-$[A_{i}';\,M_{ij}']$-bimodule, which we denote $T_{\fk{m}}$.  
Since $T_{\fk{m}}=\bigoplus_{i}T_{i\fk{r}}=\bigoplus_{j}T_{\fk{c}j}$  we refer to $T_{\fk{m}}$ as a \emph{matrix} \emph{bimodule}.

\subsection{Morita contexts for generalised matrix rings}
\label{subsec-morita-context-associativty-of-balanced-maps}
As above fix generalised Morita contexts $(A_{i};\,M_{ij};\,\varphi_{ikj})$ and $(A_{i}';\,M_{ij}';\,\varphi_{ikj}')$. 
Let $N_{\fk{m}}=(N_{ij})$ be an $[A_{i};\,M_{ij}]$-$[A_{i}';\,M_{ij}']$-matrix-bimodule as in \S\ref{subsec-morita-context-bimodules} with respect to left and right structure maps 
    \[
    \begin{array}{cc}
   \beta_{ilk}\colon M_{il}\otimes_{A_{l}}N_{lk}\to N_{ik}, 
   &
   \beta_{lkj}'\colon N_{lk}\otimes_{A_{k}} M_{kj}'\to N_{lj}.
    \end{array}
    \] 
    Likewise let $L_{\fk{m}}=(L_{ij})$ be a $[A_{i}';\,M_{ij}']$-$[A_{i};\,M_{ij}]$-matrix-bimodule with respect to 
    \[
    \begin{array}{cc}
 \gamma_{ilk}'\colon M_{il}'\otimes_{A_{l}}L_{lk}\to L_{ik},  &
   \gamma_{lkj}\colon L_{lk}\otimes_{A_{k}} M_{kj}\to L_{lj}.
    \end{array}
    \] 
     Suppose in addition the following conditions hold. 
\begin{itemize}
    \item $\alpha_{ikj}\colon N_{ik} \otimes L_{kj} \to M_{ij}$ ($i,j,k=1,\dots,n$)  are additive maps such that the map 
    \[
    \begin{array}{c}
\begin{vmatrix}
 \alpha_{1k1} & \cdots & \alpha_{1kn} \\
\vdots & \ddots & \vdots \\
\alpha_{nk1} & \cdots & \alpha_{nkn}
\end{vmatrix}
\colon
\begin{pmatrix}
N_{11} & \cdots & N_{1n} \\
\vdots & \ddots & \vdots \\
N_{n1} & \cdots & N_{nn}
\end{pmatrix}
\times 
\begin{pmatrix}
L_{11} & \cdots & L_{1n} \\
\vdots & \ddots & \vdots \\
L_{n1} & \cdots & L_{nn}
\end{pmatrix}
\to 
\begin{pmatrix}
A_{1} & \cdots & M_{1n} \\
\vdots & \ddots & \vdots \\
M_{n1} & \cdots & A_{n}
\end{pmatrix}
    \end{array}
    \]
    is an $[A_{i}';\,M_{ij}']$-balanced $[A_{i};\,M_{ij}]$-$[A_{i};\,M_{ij}]$-bimodule homomorphism.  
    \item $\alpha_{ikj}'\colon L_{ik} \otimes N_{kj} \to M_{ij}'$ are additive maps such that the map 
        \[
\begin{array}{c}
\begin{vmatrix}
 \alpha_{1k1}' & \cdots & \alpha_{1kn}' \\
\vdots & \ddots & \vdots \\
\alpha_{lkn}' & \cdots & \alpha_{nkn}'
\end{vmatrix}
\colon
\begin{pmatrix}
L_{11} & \cdots & L_{1n} \\
\vdots & \ddots & \vdots \\
L_{n1} & \cdots & L_{nn}
\end{pmatrix}
\times 
\begin{pmatrix}
N_{11} & \cdots & N_{1n} \\
\vdots & \ddots & \vdots \\
N_{n1} & \cdots & N_{nn}
\end{pmatrix}
\to 
\begin{pmatrix}
A_{1}' & \cdots & M_{1n}' \\
\vdots & \ddots & \vdots \\
M_{n1}' & \cdots & A_{n}'
\end{pmatrix}
\end{array}
\]
    is an $[A_{i};\,M_{ij}]$-balanced $[A_{i}';\,M_{ij}']$-$[A_{i}';\,M_{ij}']$-bimodule homomorphism.
\item $\beta_{ikj}(\alpha_{ihk}\otimes \myid)=\beta_{ihj}'(\myid \otimes \alpha_{hkj}')$ ($h,i,j,k=1,\dots,n$) as maps $ N_{ih}\otimes  L_{hk}\otimes N_{kj}\to N_{ij}$. 
\item $\gamma_{ihj}(\myid\otimes \alpha_{hkj})=\gamma_{ikj}'(\alpha_{ihk}'\otimes \myid)$ ($h,i,j,k=1,\dots,n$) as maps $ L_{ih}\otimes N_{hk}\otimes  L_{kj}\to L_{ij}$.  
    \end{itemize}
By the first two items above there are induced bimodule homomorphisms 
\[
\begin{array}{cc}
\left\vert
\underline{\alpha}_{ikj}
\right\vert\colon N_{\fk{m}}\otimes_{[A_{i}';\,M_{ij}']} L_{\fk{m}}\to [A_{i};\,M_{ij}], & 
\left\vert
\underline{\alpha}'_{ikj}
\right\vert\colon L_{\fk{m}}\otimes_{[A_{i};\,M_{ij}]} N_{\fk{m}}\to [A_{i}';\,M_{ij}'].
\end{array}
\]
By the second two items above the mixed associativity conditions hold for the pair of maps $\left\vert
\underline{\alpha}_{ikj}
\right\vert$ and $\left\vert
\underline{\alpha}'_{ikj}
\right\vert$, meaning there is a (classical) Morita context 
\[
\begin{array}{c}
\left([A_{i};\,M_{ij}],[A_{i}';\,M_{ij}'];\, N_{\fk{m}},L_{\fk{m}};\, 
\left\vert
\underline{\alpha}_{ikj}
\right\vert
,
\left\vert
\underline{\alpha}'_{ikj}
\right\vert
\right).
\end{array}
\]
In our main result, \Cref{thm-induced-contexts-for-generalised-matrix-rings}, we establish such a Morita context. 
In \Cref{cor-main-theorem-about-contexts-restricts-to-equivalences} we provide sufficient  conditions for $\left\vert
\underline{\alpha}_{ikj}
\right\vert$ and $\left\vert
\underline{\alpha}_{ikj}'
\right\vert$ to be surjective, in which case 
\[
\begin{array}{c}
L_{\fk{m}}\otimes_{[A_{i},M_{ij}]}-  \colon \lMod{[A_{i};\,M_{ij}]}\to  \lMod{[A_{i}';\,M_{ij}']},
\\
 N_{\fk{m}}\otimes_{[A_{i}',M_{ij}']}-  \colon \lMod{[A_{i}';\,M_{ij}']}\to  \lMod{[A_{i};\,M_{ij}]},
\end{array}
\]
 define mutually quasi-inverse functors by \Cref{thm-recalled-from-Anh-Marki}, yielding a Morita equivalence
\[
\begin{pmatrix}
A_{1} & \cdots & M_{1n} \\
\vdots & \ddots & \vdots \\
M_{n1} & \cdots & A_{n}
\end{pmatrix}
\sim
\begin{pmatrix}
A_{1}' & \cdots & M_{1n}' \\
\vdots & \ddots & \vdots \\
M_{n1}' & \cdots & A_{n}'
\end{pmatrix}.
\]

\section{Excision and ligation}
\label{sec-excision-and-ligation}

In \S\ref{sec-excision-and-ligation} we fix a particular ring-theoretic situation, and set-up the following notation.  
\begin{itemize}[$*$]
    \item $A_{1},\dots,A_{n}$, $R$ and $S$ are rings which are associative but not necessarily unital.
    \item $R=A_{t}$ as rings for some $t=1,\dots,n$ which is fixed but arbitrary.
    \item $(R,S;\,{}_{R}N_{S},{}_{S}L_{R};\,\zeta\colon N\otimes L\to R,\theta\colon L\otimes N\to S,)$ is a classical Morita context.
    \item $(A_{i};\,M_{ij};\,\varphi_{ikj}\colon M_{ik}\otimes M_{kj}\to M_{ij})$ is a generalised Morita context. 
    \item $[A_{i};\,M_{ij}]$ is the generalised matrix ring for $(A_{i};\,M_{ij};\,\varphi_{ikj})$, denoted in detail by 
    \[
    \begin{array}{c}
    [A_{i};\,M_{ij}]
    =
    \begin{bmatrix}
A_{1} & \cdots & M_{1t} & \cdots & M_{1n} \\
\vdots & \ddots & \vdots& \ddots & \vdots \\
M_{t1} & \cdots & R      & \cdots & M_{tn} \\
\vdots & \ddots & \vdots & \ddots & \vdots \\
M_{n1} & \cdots & M_{nt} & \cdots & A_{n} \\
\end{bmatrix}
    \end{array}
    \]
\end{itemize}

In Definition \ref{defn-extension-over-excision-ring} we construct a new Morita context (see \Cref{lem-extension-is-a-Morita-context-associativity})  whose  matrix ring is found by \emph{excising} entries of   $[A_{i};\, M_{ij}]$ in  row $t$ or column $t$, and \emph{ligating} the unchanged entries with substituted bimodules, such that $R$ changes to $S$.

\subsection{Excision}
\label{sec-excision}


In Definition \ref{defn-extension-over-excision-ring} we generalise the \emph{composition} of a pair of classical Morita contexts from the thesis of of Mar{\'i}n  \cite[Proposition 4.11]{Marin-1997-master-thesis-categories-and-rings} (c.f. V\"{a}ljako and Laan \cite[\S 4]{Valjako-Lann-Morita-contexts-2021}). 
\begin{defn}
\label{defn-extension-over-excision-ring}
The \emph{composition} of $(R,S;\,N,L;\,\zeta,\theta)$ and $(A_{i};\,M_{ij};\,\varphi_{ikj})$ is the tuple
\[
\begin{array}{c}
(R,S;\,N,L;\,\zeta,\theta)\circ(A_{i};\,M_{ij};\,\varphi_{ikj})=
(A_{i}';\,M_{ij}';\,\varphi_{ijk}'),
\end{array}
\] 
which is given by the rings, bimodules and bimodule homomorphisms defined below. 
\begin{itemize}[$*$]
    \item $A_{i}'$, the rings, are  given by $A_{t}'=S$ for $i=t$, and $A_{i}'=A_{i}$ for $i\neq t$. 
\item $M_{ij}'$, the $A_{i}'$-$A_{j}'$-bimodules, are  defined as follows
\[
\begin{array}{cc}
M_{tt}'=S & (i=j=t),
\vspace{1mm}
\\
M_{it}'=M_{it}\otimes_{R} N & (j=t\neq i),
\vspace{1mm}
\\
M_{tj}'=L\otimes_{R} M_{tj} &
 (i=t\neq j),
\vspace{1mm}
\\
 M_{ij}'=M_{ij} & (i,j\neq t).
\end{array}
\]
\item $\varphi_{ikj}'\colon M_{ik}'\otimes_{A_{k}'}M_{kj}'\to M_{ij}'$, the bimodule homomorphisms, are   defined as follows
\[
\begin{array}{cc}
\mu_{S}\colon S\otimes S \to S & (i=k=j=t),
\vspace{1mm}
\\
\myid_{it}^{M}\otimes \rho_{N}^{S}\colon M_{it}\otimes N\otimes S \to M_{it}\otimes N 
&
(k=j=t\neq i),  
\vspace{1mm}
\\
\lambda_{L}^{S}\otimes \myid^{M}_{tj}\colon S\otimes L\otimes M_{tj} \to L\otimes M_{tj}
&
(i=k=t\neq j),  
\vspace{1mm}
\\
\theta\circ(\rho_{L}^{R}\otimes \myid_{N})\circ(\myid_{L}\otimes \varphi_{tkt}\otimes \myid_{N})\colon L\otimes M_{tr}\otimes M_{rt}\otimes N \to S
&
(i=j=t\neq k),  
\vspace{1mm}
\\
\varphi_{ikt}\otimes \myid_{N}\colon  M_{ik}\otimes M_{kt}\otimes N \to M_{it}\otimes N
&
(j=t\neq i,k),
\vspace{1mm}
\\
\myid_{L}\otimes \varphi_{tkj}\colon  L\otimes_{R} M_{tk}\otimes M_{kj}\to L\otimes M_{tj}
&
(i=t\neq k,j), 
\vspace{1mm}
\\
\varphi_{itj}\circ(\rho_{M_{it}}^{R}\otimes \myid_{tj}^{M})\circ(\myid_{it}^{M} \otimes\zeta \otimes\myid_{tj}^{M} )\colon M_{it}\otimes  N\otimes L\otimes M_{tj}\to M_{ij} & (k=t\neq i,j),
\vspace{1mm}
\\
\varphi_{ikj}\colon M_{ik}\otimes M_{kj}\to M_{ij} & (i,j,k\neq t).
\end{array}
\]
\end{itemize}
\end{defn}

\begin{lem}
\label{lem-extension-is-a-Morita-context-associativity}
The composition $(A_{i}';\,M_{ij}';\,\varphi_{ikj}')$ of $(R,S;\,N,L;\,\zeta,\theta)$ with $(A_{i};\,M_{ij};\,\varphi_{ikj})$ is a generalised Morita context. 
Furthermore, if the rings $R,S,A_{i}$ have local units and the bimodules $N,L,M_{ij}$ are unital, then the bimodules $M_{ij}'$ are unital.
\end{lem}
\begin{proof}
By construction we have that the first four conditions in \S\ref{subsec-morita-context-ring} hold. 
To see that the composition is a generalised Morita context it suffices to check the fifth condition holds. 

Hence we begin by checking that $\varphi_{ihj}'\circ (\myid^{M\prime}_{ih}\otimes \varphi_{hkj}')=\varphi_{ikj}'\circ(\varphi_{ihk}'\otimes\myid^{M\prime}_{kj})$ holds for all $h,i,j,k=1,\dots, n$. 
When $h=i=j=k=t$ this equation follows from the associativity of multiplication in the ring $S$. 
When $t\neq h,i,j,k$ this equation follows from the fact that $(A_{i};\,M_{ij};\,\varphi_{ikj})$ is a generalised Morita context. 

We now consider the case where $h=j=t\neq i,k$. 
By definition we have
\[
\begin{array}{c}
 \varphi_{itj}'((\myid^{M\prime}_{it}\otimes \varphi_{tkj}')(m\otimes n\otimes l'\otimes m'\otimes m''))
 =
m\otimes n\theta(l'\varphi_{tkt}(m'\otimes m'')\otimes n'')
\end{array}
\]
for all $m\otimes n\in M_{it}\otimes_{R} N$ and $l'\otimes m'\in L\otimes_{R} M_{tk}$ and $m''\otimes n''\in M_{kt}\otimes_{R} N$. 
Since $\zeta$ is right $R$-linear, since $\varphi_{tkt}$ is left $R$-linear,  and by the associativity conditions for the classical Morita context $(R,S;\,N,L;\,\zeta,\theta)$, we have that $n\theta(l'\varphi_{tkt}(m'\otimes m'')\otimes n'')$ is equal to $\varphi_{tkt}(\zeta(n\otimes l')m'\otimes m'')n''$ 
which means that 
\[
\begin{array}{c}
 \varphi_{itj}'((\myid^{M\prime}_{it}\otimes \varphi_{tkj}')(m\otimes n\otimes l'\otimes m'\otimes m''))
 =
\varphi_{itt}(m\otimes \varphi_{tkt}(\zeta(n\otimes l')m'\otimes m''))\otimes n''
\end{array}
\]
which is the image of $m\otimes n\otimes l'\otimes m'\otimes m''\otimes n''$ under $\varphi_{tkj}'\circ(\varphi_{ttk}'\otimes\myid^{M\prime}_{kj})$ by definition.

Now the remaining cases are similar, and so we omit the  details.

For the second assertion in the statement, consider how the right $S$-action on $N$ defines the right action on each $M_{pt}'$ and similarly consider the left action of $S$ on each $M_{tq}'$. 
\end{proof}

By \Cref{lem-extension-is-a-Morita-context-associativity} we can and do fix the following additional notation for what remains. 
\begin{itemize}[$*$]
 \item $(A_{i}';\,M_{ij}';\,\varphi_{ikj}'\colon M_{ik}'\otimes M_{kj}'\to M_{ij}')$ is the generalised Morita context found by the composition of  $(R,S;\,N,L;\,\zeta,\theta)$ and $(A_{i};\,M_{ij};\,\varphi_{ikj})$ at $t=1,\dots,n$. 
    \item $[A_{i}';\,M_{ij}']$ is the generalised matrix ring for $(A_{i}';\,M_{ij}';\,\varphi_{ikj}')$, depicted in detail by 
    \[
    \begin{array}{c}
    [A_{i}';\,M_{ij}']
    =
    \begin{bmatrix}
A_{1} & \cdots & M_{1t}\otimes_{R} N & \cdots & M_{1n} \\
\vdots & \ddots & \vdots& \ddots & \vdots \\
L\otimes_{R}M_{t1} & \cdots & S      & \cdots & L\otimes_{R}M_{tn} \\
\vdots & \ddots & \vdots & \ddots & \vdots \\
M_{n1} & \cdots & M_{nt}\otimes_{R} N & \cdots & A_{n} \\
\end{bmatrix}
    \end{array}
    \]
\end{itemize}

\begin{defn}
\label{defn-column-excision-row-excision-additive-groups}
By the \emph{column}-\emph{excision} $(N_{ij})$ \emph{of} $(A_{i};\,M_{ij};\,\varphi_{ikj})$ \emph{by}  $(R,S;\,N,L;\,\zeta,\theta)$ \emph{at} $t=1,\dots,n$ we mean the direct sum of additive groups defined and denoted by 
 \[
    \begin{array}{c}
\begin{pmatrix}
N_{11}  & \cdots & N_{1n} \\
\vdots & \ddots & \vdots \\
N_{n1}  & \cdots & N_{nn} 
\end{pmatrix}
=
 \begin{pmatrix}
A_{1} & \cdots & M_{1t}\otimes_{R} N & \cdots & M_{1n} \\
\vdots & \ddots & \vdots& \ddots & \vdots \\
M_{t1} & \cdots & R\otimes_{R}N     & \cdots & M_{tn} \\
\vdots & \ddots & \vdots & \ddots & \vdots \\
M_{n1} & \cdots & M_{nt}\otimes_{R} N & \cdots & A_{n} 
\end{pmatrix}
    \end{array}
    \]
    That is, $(N_{ij})$ is found by tensoring of each component in column $t$ of $[A_{i};\,M_{ij}]$ with $N$. 
\end{defn}

\Cref{lem-column-excision-bimodule-structure} explains the reason for depicting the column-excision as we have. 
As seen in the proof, the bimodule is canonical in that the left and right actions correspond, as in \S\ref{subsec-morita-context-ring}--\ref{subsec-morita-context-associativty-of-balanced-maps}, to matrix multiplication with respect to the bimodule maps involved so far.

For each $j\neq t$, the additive group given by column $j$ in $(N_{ij})$ is equal to that found by taking column $j$ in the regular module for $[A_{i};\,M_{ij}]$. 
In \Cref{lem-column-excision-bimodule-structure} we equip such a group the inherited left $[A_{i};\,M_{ij}]$-module structure given by multiplication in $[A_{i};\,M_{ij}]$. 

For column $t$ in $(N_{ij})$, we note that this left module is essentially constructed in items (i) and (iii) from \cite[pp 1490--1491]{Green-Psaroudakis-artin-algebras-2014} in case $n=2$. 
Since the rings and modules we consider need not be unital, we recall this construction for completeness, and repeat the check that the given action is a left module action.

\begin{lem}
\label{lem-column-excision-bimodule-structure}
The column-excision $(N_{ij})$ 
has the following additional (canonical) structure.
\begin{enumerate}[(i)]
    \item $(N_{ij})$ is a left $[A_{i};\,M_{ij}]$-module \emph{(}$[A_{i};\,M_{ij}]$ is the generalised matrix ring\emph{)}.
    \item $(N_{ij})$ is a right $[A_{i}';\,M_{ij}']$-module \emph{(}$(A_{i}';\,M_{ij}';\,\varphi_{ijk}')$ is the composition\emph{)}. 
    \item $(N_{ij})$ is an $[A_{i};\,M_{ij}]$-$[A_{i}';\,M_{ij}']$-matrix-bimodule with respect to (i)  and (ii).
\end{enumerate}
\end{lem}

\begin{proof}
Firstly note that $(N_{ij})$ can be decomposed into additive groups, either consisting of columns, or consisting of rows. 

Following \S\ref{subsec-morita-context-modules} we show, in parts (i) and (ii) respectively, that these columns are in fact column collections for $[A_{i};\,M_{ij}]$ and row collections for $[A_{i}';\,M_{ij}']$. 
Showing this shall respectively verify the first two conditions from \S\ref{subsec-morita-context-bimodules}. 
Noting that the third condition holds by construction, in part (iii) it will be sufficient to check the fourth condition holds. 

(i) For each $j\neq t$ let $N_{ij}=M_{ij}$ and consider the column $N_{\fk{c} j}=(N_{1j},\dots,N_{nj})$ equipped with the left structure maps $\beta_{ikj}=\varphi_{ikj}$. This gives $N_{\fk{c}j}$ the structure of a left $[A_{i};\,M_{ij}]$-module. 
We now deal with column $t$, which is more involved.

So, consider $N_{\fk{c}t}=(N_{1t},\dots , N_{nt})$ where $N_{it}=M_{it}\otimes_{R}N$ for each $i$, depicted as a column $(M_{1t}\otimes_{R}N , \dots, R\otimes_{R}N,\dots, M_{nt}\otimes_{R}N)$ of additive groups. 
Define the left structure maps  
$\beta_{ikt} \colon M_{ik}\otimes_{A_{k}} N_{kt}\to N_{it}$ ($i,k=1,\dots,n$) by $\beta_{ikt}=\varphi_{ikt}\otimes \myid_{N}$. 

Again by the associativity conditions for $(A_{i};\,M_{ij};\,\varphi_{ikj})$,  we have 
\[
\begin{split}
    \beta_{iht} (\varphi_{ikh}\otimes \myid_{h}^{X})
    =(\varphi_{iht}\otimes \myid_{N}) (\varphi_{ikh}\otimes \myid_{ht}^{M}\otimes \myid_{N})
    =(\varphi_{iht} (\varphi_{ikh}\otimes \myid_{ht}^{M}))\otimes \myid_{N}
     \\
     =(\varphi_{ikt} (\myid_{ik}^{M}\otimes \varphi_{kht}))\otimes \myid_{N}
      =(\varphi_{ikt}\otimes \myid_{N}) (\myid_{ik}^{M}\otimes \varphi_{kht}\otimes \myid_{N})
    =\beta_{ikt}(\myid^{M}_{ih}\otimes \beta_{kht}),
\end{split}
\]
which verifies the fourth condition from \S\ref{subsec-morita-context-modules}, and thus the first condition from \S\ref{subsec-morita-context-bimodules}. 

(ii) As in the proof of (i) above, we need only consider the additive groups coming from row $t$ since the other rows are identical to the corresponding row in the generalised matrix ring $[A_{i}';\,M_{ij}']$.  
That is, we are letting $\beta_{ikj}'=\varphi_{ikj}'$  whenever $i\neq t$. 

So, now consider the direct sum  of the $R$-$A_{i}$-bimodules $N_{ti}$ given by $N_{ti}=M_{ti}$ for $i\neq t$ and $N_{tt}=R\otimes N$. 
Define the right structure maps $\beta_{thi}'\colon N_{th}\otimes_{A_{h}'}M_{hi}'\to N_{ti}$ by
\[
\begin{array}{cc}
\myid_{R}\otimes \rho_{N}^{S}\colon
R\otimes N\otimes_{S}S\to R\otimes N & 
(h=t=i),
\\
\lambda_{M_{ti}}^{R}\circ(\mu_{R}\otimes \myid^{M}_{ti})\circ(\myid_{R}\otimes \zeta\otimes \myid^{M}_{ti})\colon
R\otimes_{R}N\otimes_{S}L\otimes_{R}M_{ti}\to M_{ti} 
& 
(h=t\neq i),
\\
\varphi_{tht}\otimes \myid_{N}\colon 
M_{th}\otimes_{A_{h}}M_{ht}\otimes_{R}N\to R\otimes_{R} N   
& 
(h\neq t=i),
\\
\varphi_{thi}\colon 
M_{th}\otimes_{A_{h}}M_{hi}\to M_{ti}
&
(h\neq t\neq i).
\end{array}
\]
We check that the equation $\beta_{thj}' (\myid_{th}^{Y}\otimes \varphi_{hij}')
   =
\beta_{tij}'(\beta_{thi}'\otimes \myid_{ij}^{M\prime})$  holds for the case $h=j=t\neq i$. 
The cases where $h=t\neq i,j$ and $i=t\neq h,j$ are similar, and the remaining cases are more straightforward, and so such details are omitted. 

Begin by noting that for any $r\in R$, $n,n'\in N$, $l\in L$, $m\in M_{it}$ and $m'\in M_{ti}$ we have 
\[
 \begin{split}
  r\otimes n\theta(l\varphi_{tit}(m\otimes m')\otimes n')
   =
  r\otimes \zeta(n\otimes l\varphi_{tit}(m\otimes m')) n'
   \\
   =
  r\otimes \zeta(n\otimes l)\varphi_{tit}(m\otimes m') n'
   =
  r\otimes \varphi_{tit}(\zeta(n\otimes l)m\otimes m') n'
   \\
   =
  r\varphi_{tit}(\zeta(n\otimes l)m\otimes m')\otimes  n'
   =
  \varphi_{tit}(r\zeta(n\otimes l)m\otimes m')\otimes  n'
  \\
   =
  \varphi_{tit}(\lambda_{M_{ti}}^{R}(\mu_{R}(r\otimes \zeta(n\otimes l))\otimes m)\otimes m')\otimes  n'.
 \end{split}
\]
As a result we have the following equalities,
\[
\begin{split}
   (\myid_{R}\otimes \rho_{N}^{S}) (\myid_{tt}^{Y}\otimes \varphi_{tit}')
    =
    (\myid_{R}\otimes \rho_{N}^{S}) (\myid_{tt}^{Y}\otimes (\theta(\rho_{L}^{R}\otimes \myid_{N})(\myid_{L}\otimes \varphi_{tit}\otimes \myid_{N})))
    \\
    =
    (\myid_{R}\otimes \rho_{N}^{S}) ((\myid_{R}\otimes\myid_{N})\otimes (\theta(\rho_{L}^{R}\otimes \myid_{N})(\myid_{L}\otimes \varphi_{tit}\otimes \myid_{N})))
\\
=
(\varphi_{tit}(\lambda_{M_{ti}}^{R}(\mu_{R}\otimes \myid^{M}_{ti})(\myid_{R}\otimes \zeta\otimes \myid^{M}_{ti}))\otimes \myid_{it}^{M})\otimes\myid_{N}
    =    (\varphi_{tit}\otimes \myid_{N})(\beta_{ti}'\otimes \myid_{it}^{M}\otimes\myid_{N}).
\end{split}
\]
Thus we have $\beta_{ttt}' (\myid_{tt}^{Y}\otimes \varphi_{tit}')
   =
\beta_{tit}'(\beta_{tti}'\otimes \myid_{it}^{M\prime})$, which  verifies the dual of the fourth condition from \S\ref{subsec-morita-context-modules} (for the case we cover), and thus the second condition from \S\ref{subsec-morita-context-bimodules}. 

(iii) As discussed above, the third condition from \S\ref{subsec-morita-context-bimodules} is automatic and it is sufficient from here to check the fourth condition holds. 
Hence we will show $\beta_{ilk} 
(\myid_{il}^{M}\otimes
\beta_{lkj}')
=
\beta_{ikj}' 
(\beta_{ilk}\otimes
\myid_{kj}^{M\prime})$ for all $i,j,k,l=1,\dots,n$. 
First consider cases based on which elements of $\{i,j,k ,l\}$ are equal to $t$. 
The proof in the case $k,l=t\neq i,j$ follows from similar arguments to those used in the proofs of (i) and (ii) above, noting that we have
\[
\begin{array}{cc}
\varphi_{ihj} (\myid^{M}_{ih}\otimes \varphi_{hkj})=\varphi_{ikj}(\varphi_{ihk}\otimes\myid^{M}_{kj}),
&
\varphi_{ihj}' (\myid^{M\prime}_{ih}\otimes \varphi_{hkj}')=\varphi_{ikj}'(\varphi_{ihk}'\otimes\myid^{M\prime}_{kj})
\end{array}
\]
by the assumption that $(A_{i};\,M_{ij};\,\varphi_{ikj})$ is a generalised Morita context, and by \Cref{lem-extension-is-a-Morita-context-associativity}. 
The remaining cases either again follow from similar arguments as used in (i) and (ii), or are more straightforward, and so we omit them. 
\end{proof}

\begin{defn}
\label{defn-row-excision-row-excision-additive-groups}
By the \emph{row}-\emph{excision} $(L_{ij})$ 
we mean the direct sum of additive groups defined and denoted by 
 \[
    \begin{array}{c}
\begin{pmatrix}
L_{11}  & \cdots & L_{1n} \\
\vdots & \ddots & \vdots \\
L_{n1}  & \cdots & L_{nn} 
\end{pmatrix}
=
\begin{pmatrix}
A_{1} & \cdots & M_{1t} & \cdots & M_{1n} \\
\vdots & \ddots & \vdots& \ddots & \vdots \\
L\otimes_{R}M_{t1} & \cdots & L\otimes_{R}R     & \cdots & L\otimes_{R}M_{tn} \\
\vdots & \ddots & \vdots & \ddots & \vdots \\
M_{n1} & \cdots & M_{nt} & \cdots & A_{n} \\
\end{pmatrix}.
    \end{array}
    \]
That is, $(L_{ij})$ is found by tensoring of each component in row $t$ of $[A_{i};\,M_{ij}]$ with $L$.
\end{defn}
Dually to \Cref{lem-column-excision-bimodule-structure}, we have the following. 
\begin{lem}
\label{lem-row-excision-bimodule-structure}
The row-excision $(L_{ij})$ 
has the following additional (canonical) structure.
\begin{enumerate}[(i)]
    \item $(L_{ij})$ is a left $[A_{i}';\,M_{ij}']$-module.
    \item $(L_{ij})$ is a right $[A_{i};\,M_{ij}]$-module.
    \item $(L_{ij})$ is an $[A_{i}';\,M_{ij}']$-$[A_{i};\,M_{ij}]$-matrix-bimodule with respect to (i)  and (ii).
\end{enumerate}
\end{lem}

\begin{proof}
For each $i=1,\dots, n$ define $L_{i\fk{r}}=(L_{i1},\dots,L_{in})$ where $L_{tj}=L\otimes M_{tj}$ and $L_{ij}=M_{ij}$ for $i\neq t$. 
Define the right structure maps $\gamma_{ikj}\colon L_{ik}\otimes_{A_{k}}M_{kj}\to L_{ij}$ by $\gamma_{tkj}=\myid_{L}\otimes \varphi_{tkj}$ and $\gamma_{ikj}= \varphi_{ikj}$ for $i\neq t$. 
For each $j=1,\dots, n$ define $L_{\fk{c}j}=(L_{1j},\dots,L_{nj})$ where  $L_{tj}=L\otimes M_{tj}$ and $L_{ij}=M_{ij}$ for $i\neq t$. 
Hence $X=Y$ which is the additive group $(L_{ij})$. 
Define the left structure maps  $\gamma_{ikj}' \colon M_{ik}'\otimes_{A_{k}} L_{kj}\to L_{ij}$ ($i,j,k=1,\dots,n$) by  $\gamma_{ikj}'=\varphi_{ikj}'$ for $j\neq t$, and otherwise as follows. 
Define the maps $\gamma_{ikt}' \colon M_{ik}'\otimes_{A_{k}'} L_{kt}\to L_{it}$ ($i,k=1,\dots,n$) by 
\[
\begin{array}{cc}
\lambda_{L}^{S}\otimes\myid_{R}
\colon
S\otimes_{S} L\otimes_{R} R\to L\otimes_{R}R & 
(i=t=k),
\\
\myid_{L}\otimes\varphi_{tkt}\colon
L\otimes_{R} M_{tk}\otimes_{A_{k}} M_{kt}\to L\otimes_{R} R
& 
(i=t\neq k),
\\
\rho_{M_{it}}^{R}\circ (\myid_{it}^{M}\otimes \mu_{R})\circ (\myid_{it}^{M}\otimes \zeta\otimes \myid_{R})\colon 
M_{it}\otimes_{R}N \otimes_{S} L\otimes_{R} R\to M_{it}
& 
(i\neq t=k),
\\
\varphi_{ikt}\colon 
M_{ik}\otimes_{A_{k}} M_{kt}\to M_{it}
&
(i\neq t\neq k).
\end{array}
\]
We leave the verification of (i)--(iii) to the reader, since such  arguments are similar to those used in the proof of \Cref{lem-column-excision-bimodule-structure}. 
Note that $(L_{ij})=\bigoplus_{i}L_{i\fk{r}}$ as right $[A_{i};\,M_{ij}]$-modules and $(L_{ij})=\bigoplus_{j}L_{\fk{c}j}$ as left $[A_{i}';\,M_{ij}']$-modules. 
\end{proof}

\subsection{Ligation}
\label{sec-ligation}

By Lemmas \ref{lem-column-excision-bimodule-structure} and \ref{lem-row-excision-bimodule-structure} the column-excision $(N_{ij})$ and the row-excision $(L_{ij})$  define matrix bimodules over the generalised matrix rings $[A_{i};\,M_{ij}]$ and $[A_{i}';\,M_{ij}']$ of the generalised Morita contexts $(A_{i};\,M_{ij};\,\varphi_{ikj})$ and $(A_{i}';\,M_{ij}';\,\varphi_{ikj}')$. 

From now on we relabel using the notation $N_{\fk{m}}=(N_{ij})$ and  $L_{\fk{m}}=(L_{ij})$ to keep track of the fact these additive groups are matrix bimodules.  

\begin{defn}
\label{defn-column-row-ligation}
The \emph{column}-\emph{row ligation}  is the biadditive map $\left\vert
\alpha_{ikj}
\right\vert$ of the form
\[
\begin{array}{c}
\begin{vmatrix}
 \alpha_{1k1} & \cdots & \alpha_{1kn} \\
\vdots & \ddots & \vdots \\
\alpha_{nk1} & \cdots & \alpha_{nkn}
\end{vmatrix}
\colon
\begin{pmatrix}
N_{11} & \cdots & N_{1n} \\
\vdots & \ddots & \vdots \\
N_{n1} & \cdots & N_{nn}
\end{pmatrix}
\times 
\begin{pmatrix}
L_{11} & \cdots & L_{1n} \\
\vdots & \ddots & \vdots \\
L_{n1} & \cdots & L_{nn}
\end{pmatrix}
\to 
\begin{pmatrix}
A_{1} & \cdots & M_{1n} \\
\vdots & \ddots & \vdots \\
M_{n1} & \cdots & A_{n}
\end{pmatrix},
\end{array}
\]
given by the $A_{i}$-$A_{j}$-bimodule homomorphisms $\alpha_{ikj}\colon N_{ik}\otimes_{A_{k}'} L_{kj}\to M_{ij}$ defined by
\[
\begin{array}{cc}
\varphi_{ikj}\colon M_{ik}\otimes_{A_{k}} M_{kj}\to M_{ij}   &  (k\neq t),\\
\varphi_{itj}\circ(\rho_{M_{it}}^{R}\otimes \myid_{tj}^{M})\circ(\myid_{it}^{M} \otimes\zeta \otimes\myid_{tj}^{M} )\colon M_{it}\otimes_{R} N\otimes_{S} L\otimes_{R} M_{tj}\to M_{ij}  
 &  (k=t).
\end{array}
\]
\end{defn}



In what follows we show that the column-row ligation is balanced, and hence factors through the corresponding tensor product; see Lemma  \ref{lem-first-bilinear-forms-between-excised-balanced}.

\begin{lem}
\label{lem-first-bilinear-forms-between-excised-balanced}
The column-row-ligation $\left\vert\alpha_{ikj}\right\vert\colon N_{\fk{m}}\times L_{\fk{m}}\to [A_{i};\,M_{ij}]$ is $[A_{i}';\,M_{ij}']$-balanced. 
\end{lem}

\begin{proof}
Recall the right structure maps $\beta_{ihk}'\colon N_{ih}\otimes_{A_{h}'} M_{hk}'\to N_{ik}$ from \Cref{lem-column-excision-bimodule-structure}(ii) and the left structure maps $\gamma_{ijl}'\colon M_{ij}'\otimes_{A_{j}'}L_{jl}\to L_{il}$ from \Cref{lem-row-excision-bimodule-structure}(i). 
Considering the final condition from \ref{subsec-morita-context-balanced-maps-of-modules}, to show the given map is balanced we require that $\alpha_{ihj}(\myid_{ih}^{N} \otimes \gamma_{hkj}')
=\alpha_{ikj}(\beta_{ihk}'\otimes \myid_{kj}^{L})$ ($h,i,j,k=1,\dots,n$) as maps $N_{ih}\otimes M_{hk}'\otimes L_{kj}\to M_{ij}$. 
As was done in the proofs of the aforementioned lemmas, we break our claim down into cases. 

When $h\neq t\neq k$ the required equation follows from the mixed associativity conditions held by the generalised Morita context $(A_{i};\,M_{ij};\,\varphi_{ikj})$. 
When $i\neq t\neq j$ we can similarly apply the conditions held by the context $(A_{i}';\,M_{ij}';\,\varphi_{ikj}')$, which hold by  \Cref{lem-extension-is-a-Morita-context-associativity}. 
When $h=t=k$ or $i=t=j$ the required equation is straightforward. There are four cases left: when $h,i=t\neq j,k$; when $h,j=t\neq i,k$; when $k,i=t\neq h,j$; and when $k,j=t\neq h,i$. 
We show the required equation holds in the first two cases, since the arguments for the other two cases are similar. 
So suppose  $h,i=t\neq j,k$. 
The required equation then becomes
\[
(\lambda_{M_{tj}}^{R}(\mu_{R}\otimes \myid)(\myid_{R} \otimes\zeta \otimes\myid ))
(
\myid\otimes \myid
\otimes 
\myid\otimes \varphi_{tkj}
)
=\varphi_{tkj}
(
(\lambda_{M_{tk}}^{R}(\mu_{R}\otimes \myid)(\myid\otimes \zeta\otimes \myid))
\otimes
\myid
)
\]
considered as maps of the form $R\otimes N\otimes L\otimes M_{tk}\otimes M_{kj}\to M_{tj}$. 
The proof from here is straightforward, since it can be shown that both sides send any pure tensor $r\otimes n\otimes l\otimes m\otimes m'$ to $r\zeta(n\otimes l)\varphi_{tkj}( m\otimes m')$ using the linearity of the involved homomorphisms. 

Suppose instead $h,j=t\neq i,k$. 
This time the required equation becomes 
\[
\rho_{M_{it}}^{R}(\rho_{M_{it}}^{R}\otimes\myid)(\myid \otimes\zeta \otimes\myid )
(
\myid\otimes\myid
\otimes 
\myid\otimes\varphi_{tkt}
)
=\varphi_{ikt}
(
(\varphi_{itk}(\rho_{M_{it}}^{R}\otimes \myid)(\myid \otimes\zeta \otimes\myid ))
\otimes
\myid
)
\] 
considered as maps of the form $M_{it}\otimes N\otimes L\otimes M_{tk}\otimes M_{kt}\to M_{it}$. 
In this case one can show both sides send a pure tensor $m\otimes n\otimes l\otimes m'\otimes m''$ to $m \zeta(n\otimes l)\varphi_{tkt}(m'\otimes m'')$,  using that $\varphi_{ikt}(\varphi_{itk}\otimes \myid)=\rho_{M_{it}}^{R}(\myid\otimes \varphi_{tkt})$ by the associativity conditions for $(A_{i};\,M_{ij};\,\varphi_{ikj})$. 
\end{proof}

Note that the bimodule $M_{ii}$ is unital if and only if the map $\mu_{A_{i}}\colon A\otimes A\to A$ is surjective. 

\begin{lem}
\label{lem-column-row-ligation-gives-bimodule-homomorphism-with-surjectivity-criteria}
The canonical additive map $\left\vert\underline{\alpha}_{ikj}\right\vert\colon N_{\fk{m}}\otimes_{[A_{i}';\,M_{ij}']} L_{\fk{m}}\to [A_{i};\,M_{ij}]$, defined by the column-row ligation, is an $[A_{i};\,M_{ij}]$-$[A_{i};\,M_{ij}]$-bimodule homomorphism. 
Furthermore, $\left\vert\underline{\alpha}_{ikj}\right\vert$ is surjective provided $\zeta$ is surjective and each $A_{i}$-$A_{j}$-bimodule $M_{ij}$ is unital. 
\end{lem}

\begin{proof} 
Our first assertion is that $\left\vert\underline{\alpha}_{ikj}\right\vert$ is a left $[A_{i};\,M_{ij}]$-module homomorphism. 
Recall the left structure maps $\beta_{ihk}\colon M_{ih}\otimes_{A_{h}} N_{hk}\to N_{ik}$ from \Cref{lem-column-excision-bimodule-structure}(i).  
As in \S\ref{subsec-morita-context-homomorphisms-of-modules}, for our first assertion we will show that $\alpha_{ikj}(\beta_{ihk}\otimes \myid_{kj}^{L})=\varphi_{ihj}(\myid_{ih}^{M}\otimes \alpha_{ikj})$ for all $h,i,j,k=1,\dots,n$. 

When $k\neq t$ this equality follows from the mixed associativity conditions held by $(A_{i};\,M_{ij};\,\varphi_{ikj})$. 
When $k= t$, the required equality becomes 
\[
\varphi_{itj}(\rho_{M_{it}}^{R}\otimes \myid_{tj}^{M})(\myid_{it}^{M}\otimes \zeta\otimes \myid_{tj}^{M})(\varphi_{iht}\otimes\myid_{N}\otimes\myid_{L}\otimes\myid_{tj}^{M})
=
\varphi_{ihj}(\myid_{ih}^{M}\otimes \varphi_{htj}(\rho_{M_{ht}}^{M}\otimes \myid_{tj}^{M})(\myid_{ht}^{M}\otimes\zeta\otimes\myid_{tj}^{M}))
\]
which follows from the equality $\varphi_{ihj}(\myid_{it}^{M}\otimes\varphi_{htj})=\varphi_{itj}(\varphi_{iht}\otimes \myid_{tj}^{M})$, which as above is just one of the mixed associativity conditions held by $(A_{i};\,M_{ij};\,\varphi_{ikj})$. 

Our second assertion is that $\left\vert\underline{\alpha}_{ikj}\right\vert$ is a right $[A_{i};\,M_{ij}]$-module homomorphism.  
So instead consider the right structure maps $\gamma_{ikj}\colon L_{ik}\otimes_{A_{k}}M_{kj}\to L_{ij}$ from \Cref{lem-row-excision-bimodule-structure}(ii). 
For our second assertion we show that $\alpha_{ihj}(\myid_{ih}^{N}\otimes \gamma_{hkj})=\varphi_{ikj}(\alpha_{ihk}\otimes \myid_{kj}^{M})$ for all $h,i,j,k$.

As above, the proof that this equality holds is straightforward. 
In case $h=t$ the equation follows automatically as a consequence of the mixed associativity conditions for $(A_{i};\,M_{ij};\,\varphi_{ikj})$. 
In case $h\neq t$, it follows from the equation $\varphi_{itj}(\myid_{it}^{M}\otimes \varphi_{tkj})=\varphi_{ikj}(\varphi_{itk}\otimes \myid^{M}_{kj})$. 

For the final claim in the statement, it suffices to fix $i$ and $j$ arbitrary and choose some $k$ such that the map $\alpha_{ikj}\colon N_{ik}\otimes_{A_{k}'}L_{kj}\to M_{ij}$ is surjective.

Since the $A_{i}$-$A_{j}$ bimodules $M_{ij}$ are unital the maps $\rho_{M_{ik}}^{A_{k}}=\varphi_{ikk}$  and $\lambda_{M_{kj}}^{A_{k}}=\varphi_{kkj}$ are surjective. 
When $k\neq t$ we have $\alpha_{ikj}=\varphi_{ikj}$. 
Hence if $i\neq t$ then  it suffices to choose $k=i$, and if $j\neq t$ it suffices to choose $k=j$. 
So we can assume $i=t=j$, and we choose $k=t$, so that $\alpha_{ikj}$ is the composition of $\varphi_{ttt}$, $\rho_{M_{tt}}^{R}\otimes \myid$ and $\myid\otimes\zeta\otimes\myid$. 
As above, $\varphi_{ttt}=\mu_{R}=\rho_{M_{tt}}^{R}$ is surjective, and $\zeta$ is surjective by assumption. 
\end{proof}

Later we use \Cref{lem-column-row-ligation-gives-bimodule-homomorphism-with-surjectivity-criteria}, together with results from discussed in the sequel, in conjunction with \Cref{thm-recalled-from-Anh-Marki}. 
We start with a dual construction to the column-row ligation from Definition \ref{defn-column-row-ligation}, completing the ingredients needed to state \Cref{thm-induced-contexts-for-generalised-matrix-rings}.

\begin{defn}
\label{row-column-ligation}
The \emph{row}-\emph{column ligation} is the biadditive map $\left\vert
\alpha_{ikj}'
\right\vert$ of the form
\[
\begin{array}{c}
\begin{vmatrix}
 \alpha_{1k1}' & \cdots & \alpha_{1kn}' \\
\vdots & \ddots & \vdots \\
\alpha_{nk1}' & \cdots & \alpha_{nkn}'
\end{vmatrix}
\colon
\begin{pmatrix}
L_{11} & \cdots & L_{1n} \\
\vdots & \ddots & \vdots \\
L_{n1} & \cdots & L_{nn}
\end{pmatrix}
\times \begin{pmatrix}
N_{11} & \cdots & N_{1n} \\
\vdots & \ddots & \vdots \\
N_{n1} & \cdots & N_{nn}
\end{pmatrix}
\to 
\begin{pmatrix}
A_{1}' & \cdots & M_{1n}' \\
\vdots & \ddots & \vdots \\
M_{n1}' & \cdots & A_{n}'
\end{pmatrix},
\end{array}
\]
given by the $A_{i}'$-$A_{j}'$-bimodule homomorphisms $\alpha_{ikj}'\colon L_{ik}\otimes_{A_{k}} N_{kj}\to M_{ij}'$ defined by 
\[
\begin{array}{cc}
\varphi_{ikj}\colon 
M_{ik}\otimes_{A_{k}} M_{kj}\to M_{ij}
&  
(i\neq t\neq j),
\\
\varphi_{ikt}\otimes\myid_{N} \colon M_{ik}\otimes_{A_{k}} M_{kt}\otimes_{R}N\to M_{it}\otimes_{R}N
&  
(i\neq t=j),
\\
\myid_{L}\otimes\varphi_{tkj} \colon L\otimes_{R}M_{tk}\otimes_{A_{k}} M_{kj}\to L\otimes_{R}M_{tj}
&  
(i=t\neq j),
\\
\theta\circ(\rho_{L}^{R}\otimes \myid_{N})\circ(\myid_{L}\otimes \varphi_{tkt}\otimes \myid_{N})\colon L\otimes_{R}M_{tk}\otimes_{A_{k}} M_{kt}\otimes_{R}N\to S 
&  (i=t=j).
\end{array}
\]
\end{defn}

\begin{lem}
\label{lem-second-bilinear-forms-between-excised-balanced}
The row-column ligation $\left\vert
\alpha_{ikj}'
\right\vert\colon L_{\fk{m}}\times N_{\fk{m}}\to [A_{i}';\,M_{ij}']$ is $[A_{i};\,M_{ij}]$-balanced.   
\end{lem}

\begin{proof}
We follow the strategy from the proof of \Cref{lem-first-bilinear-forms-between-excised-balanced}. 
Recall the left structure maps $\beta_{hkj} \colon M_{hk}\otimes_{A_{k}} N_{kj}\to N_{hj}$
from \Cref{lem-column-excision-bimodule-structure}(i), 
and the right structure maps $\gamma_{ihk}\colon L_{ih}\otimes_{A_{h}}M_{hk}\to L_{ik}$ 
from \Cref{lem-row-excision-bimodule-structure}(ii). 
As in \Cref{lem-first-bilinear-forms-between-excised-balanced}, but dually, it suffices check that $\alpha_{ihj}'(\myid_{ih}^{L} \otimes \beta_{hkj})
=\alpha_{ikj}'(\gamma_{ihk}\otimes \myid_{kj}^{N})$ as maps $L_{ih}\otimes M_{hk}\otimes N_{kj}\to M_{ij}'$. 
To do so is particularly straightforward, and only uses the associativity conditions for $(A_{i};\,M_{ij};\,\varphi_{ikj})$.
\end{proof}

\begin{lem}
\label{lem-row-column-ligation-gives-bimodule-homomorphism-with-surjectivity-criteria}
The canonical additive map $\left\vert\underline{\alpha}_{ikj}'\right\vert\colon L_{\fk{m}}\otimes_{[A_{i};\,M_{ij}]} N_{\fk{m}}\to [A_{i}';\,M_{ij}']$, defined by the row-column ligation, is an $[A_{i}';\,M_{ij}']$-$[A_{i}';\,M_{ij}']$-bimodule homomorphism. 
Furthermore, $\left\vert\underline{\alpha}_{ikj}'\right\vert$ is surjective provided $\theta$ is surjective and each $A_{i}$-$A_{j}$-bimodule $M_{ij}$ is unital.  
\end{lem}

\begin{proof}

Recall the  left and right structure maps $\gamma_{ijl}'\colon M_{ij}'\otimes_{A_{j}'}L_{jl}\to L_{il}$ and   $\beta_{ihk}'\colon N_{ih}\otimes_{A_{h}'} M_{hk}'\to N_{ik}$ from Lemmas \ref{lem-row-excision-bimodule-structure}(i) and \ref{lem-column-excision-bimodule-structure}(ii).  
As in the proof of \Cref{lem-column-row-ligation-gives-bimodule-homomorphism-with-surjectivity-criteria}, to show $\left\vert\underline{\alpha}_{ikj}'\right\vert$ is a left $[A_{i}';\,M_{ij}']$-module homomorphism it suffices to show $\alpha_{ikj}'(\gamma_{ihk}'\otimes \myid_{kj}^{N})=\varphi_{ihj}'(\myid_{ih}^{M\prime}\otimes \alpha_{hkj}')$. 

The proof that this equation holds is similar to the arguments seen so far in the article. Namely, consider cases, and  when appropriate, apply the mixed associativity conditions for $(A_{i};\,M_{ij};\,\varphi_{ikj})$ or apply \Cref{lem-extension-is-a-Morita-context-associativity}.  
The proof that $\left\vert\underline{\alpha}_{ikj}'\right\vert$ is a right $[A_{i}';\,M_{ij}']$-module homomorphism follows a similar procedure. 

The proof that $\left\vert\underline{\alpha}_{ikj}'\right\vert$ is surjective given the stated provisions hold follows similar arguments to those in the proof of the second statement in \Cref{lem-column-row-ligation-gives-bimodule-homomorphism-with-surjectivity-criteria}, and omitted. 
\end{proof}


\subsection{Corner replacement}
\label{subsec-corner-replace}

We conclude \S\ref{sec-excision-and-ligation} by showing the canonical additive maps considered in Lemmas \ref{lem-column-row-ligation-gives-bimodule-homomorphism-with-surjectivity-criteria} and \ref{lem-row-column-ligation-gives-bimodule-homomorphism-with-surjectivity-criteria} satisfy mixed associativity conditions, and hence give rise to a classical Morita context of matrix rings and matrix bimodules;  see \Cref{thm-induced-contexts-for-generalised-matrix-rings}.

The title of \S\ref{subsec-corner-replace} refers to the following `surgery' to be performed on $(A_{i};\,M_{ij};\,\varphi_{ikj})$. 
\begin{itemize}
    \item `Cut out' the ring $R$ and the bimodules $M_{ij}$ ($i=t\neq j$ or $j=t\neq i$)   using the column-excision and row-excision matrix bimodules from \S\ref{sec-excision}. 
    \item `Graft back' a new context from the column-row and row-column ligations in \S\ref{sec-ligation}.
    \item Using the items above we `corner replace' $R$ in $[A_{i};\,M_{ij}]$ with $S$ to obtain $[A_{i}';\,M_{ij}']$.
\end{itemize}  

\begin{lem}
\label{lem-mixed-associativity-conditions-between-excised-matrix-bimodules}
The maps $\left\vert\underline{\alpha}_{ikj}\right\vert$ and $\left\vert\underline{\alpha}_{ikj}'\right\vert$ satisfy the mixed associativity conditions.
\end{lem}

\begin{proof}
Considering the conditions from \S\ref{subsec-morita-context-associativty-of-balanced-maps}, we need to check the equations in the final two items hold. 
That is, we need to check that the equations 
\[
\begin{array}{cc}
(*) & \beta_{ikj}(\alpha_{ihk}\otimes \myid_{kj}^{N})=\beta_{ihj}'(\myid_{ih}^{N} \otimes \alpha_{hkj}')
,
\\
(**) &
\gamma_{ihj}(\myid_{ih}^{L}\otimes \alpha_{hkj})=\gamma_{ikj}'(\alpha_{ihk}'\otimes \myid_{kj}^{L}),
\end{array}
\]
hold for all $h,i,j,k=1,\dots,n$. 
As usual we break the proof into cases. 

When $h\neq t\neq j$ the required equality $(*)$ follows directly from the mixed associativity conditions for $(A_{i};\,M_{ij};\,\varphi_{ikj})$. Similarly when $i\neq t\neq k$ the equality $(**)$ follows directly. 

When $h\neq t=j$ we can rewrite $(*)$ as  $\sigma\otimes \myid_{N}=\tau\otimes \myid_{N}$ for $A_{i}$-$R$-bimodule homomorphisms  $\sigma$ and $\tau$, and  the aforementioned associativity conditions  give $\sigma=\tau$. 
Likewise when $i=t\neq k$ there exist $R$-$A_{j}$ bimodule homomorphisms $\sigma'$ and $\tau'$ where $(**)$ has the form $\myid_{L}\otimes\sigma'=\myid_{L}\otimes\tau'$ and where  $\sigma'=\tau'$ can be seen following a similar argument. 

We complete the proof by showing $(*)$ and $(**)$ hold in the case where $h=i=k=t\neq j$. 
All the cases that remain for checking $(*)$ holds, and all those that remain for $(**)$, can be seen using similar arguments: and hence we omit such details. 

So, suppose $h,i,k=t\neq j$.  
Here we have  
\[
\begin{array}{ccc}
 \beta_{ikj}=\varphi_{ttj}, & 
 \alpha_{ihk}=\varphi_{ttt}(\rho_{M_{tt}}^{R}\otimes \myid_{tt}^{M})(\myid_{tt}^{M} \otimes\zeta \otimes\myid_{tt}^{M}),\\
 \myid_{kj}^{N}=\myid_{tj}^{M},
 &
 \myid_{ih}^{N}=\myid_{tt}^{M}\otimes \myid_{N},\\
 \beta_{ihj}'=\lambda_{M_{tj}}^{R}(\mu_{R}\otimes \myid^{M}_{tj})(\myid_{R}\otimes \zeta\otimes \myid^{M}_{tj}), &
 \alpha_{hkj}'=\myid_{L}\otimes\varphi_{ttj}.
\end{array}
\]
Consequently, substituting these values, the equation $(*)$ becomes
\[
\varphi_{ttj}
(
(\varphi_{ttt}(\rho_{M_{tt}}^{R}\otimes \myid_{tt}^{M})(\myid_{tt}^{M} \otimes\zeta \otimes\myid_{tt}^{M} ))
\otimes 
\myid_{tj}^{M}
)
=
\lambda_{M_{tj}}^{R}(\mu_{R}\otimes \myid^{M}_{tj})(\myid_{R}\otimes \zeta\otimes \myid^{M}_{tj})
(
\myid_{tt}^{M}\otimes \myid_{N}
\otimes 
\myid_{L}\otimes\varphi_{ttj}
).
\]
That this equality holds is straightforward, since for any $r,r'\in R$, $n\in N$, $l\in L$ and $m\in M_{tj}$ both sides send the pure tensor $r\otimes n\otimes l\otimes r'\otimes m$ to $r\zeta(n\otimes l)r'm$. 

We now check $(**)$ holds, which is more involved. For the case we are in we also have 
\[
\begin{array}{ccc}
\gamma_{ihj}=\myid_{L}\otimes \varphi_{ttj}, & 
\alpha_{hkj}=\varphi_{ttj}(\rho_{M_{tt}}^{R}\otimes \myid_{tj}^{M})(\myid_{tt}^{M} \otimes\zeta \otimes\myid_{tj}^{M} ), \\
\myid_{ih}^{L}=\myid_{L}\otimes \myid_{tt}^{M},
 &
 \gamma_{ikj}'=\varphi_{ttj}'=\lambda_{L}^{S}\otimes \myid^{M}_{tj}, \\
 \alpha_{ihk}'=\theta(\rho_{L}^{R}\otimes \myid_{N})(\myid_{L}\otimes \varphi_{ttt}\otimes \myid_{N}), &
 \myid_{kj}^{L}=\myid_{L}\otimes \myid_{tj}^{M}.
\end{array}
\]
As was done above by substituting these values into $(**)$ we require the equation
\[
\myid_{L}\otimes (\varphi_{ttj}
(\myid_{tt}^{M}
\otimes 
(\varphi_{ttj}(\rho_{M_{tt}}^{R}\otimes \myid_{tj}^{M})(\myid_{tt}^{M} \otimes\zeta \otimes\myid_{tj}^{M} ))
))
=
(\lambda_{L}^{S}
\theta(\rho_{L}^{R}\otimes \myid_{N})(\myid_{L}\otimes \varphi_{ttt}\otimes \myid_{N})
\otimes 
\myid_{L}
)
\otimes \myid^{M}_{tj}.
\]
Fix  $l,l'\in L$, $r,r'\in R$, $n\in N$ and $m\in M_{tj}$. 
By using the mixed associativity conditions for $(R,S;\,N,L;\,\zeta,\theta)$, the map on the right hand side sends $l\otimes r\otimes r'\otimes n\otimes l'\otimes m$ to
\[
\begin{split}
\lambda_{L}^{S}
\left(
\theta(\rho_{L}^{R}\otimes \myid_{N})(\myid_{L}\otimes \varphi_{ttt}\otimes \myid_{N})
(l\otimes  r\otimes r'\otimes n)
\otimes l'
\right)\otimes m
\\
=
\lambda_{L}^{S}
\left(
\theta(\rho_{L}^{R}\otimes \myid_{N})
(l\otimes  r r'\otimes n)
\otimes l'
\right)\otimes m
=
\lambda_{L}^{S}
\left(
\theta
(l r r'\otimes n)
\otimes l'
\right)\otimes m
\\
=
\theta
(l r r'\otimes n)
l'
\otimes m
=
lr r'
\zeta
(n\otimes l')
\otimes m
=l
\otimes r r'
\zeta
(n\otimes l')m
\end{split}
\]
which is the image of $l\otimes r\otimes r'\otimes n\otimes l'\otimes m$ under the map on the left hand side. 
\end{proof}

We are now ready to state and prove our main result. 

\begin{thm}
\label{thm-induced-contexts-for-generalised-matrix-rings}
Let $(A_{i}';\,M_{ij}';\,\varphi_{ikj}')$ be the composition of a generalised Morita context $(A_{i};\,M_{ij};\,\varphi_{ikj})$ with a Morita context $(R,S;\,N,L;\,\theta,\zeta)$ where $R=A_{t}$ for some $t$.

In the notation so far, consider: the generalised matrix rings $[A_{i};\,M_{ij}]$ and $[A_{i}';\,M_{ij}']$; the matrix bimodules $N_{\fk{m}}$ and $L_{\fk{m}}$ given by row and column excision at $t$; and the induced maps $\left\vert\underline{\alpha}_{ikj}\right\vert\colon N_{\fk{m}}\otimes L_{\fk{m}}\to [A_{i};\,M_{ij}]$ $\left\vert\underline{\alpha}_{ikj}'\right\vert\colon L_{\fk{m}}\otimes N_{\fk{m}}\to [B_{i};\,P_{ij}]$ from the balanced maps of row and column ligation at $t$. 
Then the tuple 
\[
([A_{i};\,M_{ij}],[A_{i}';\,M_{ij}'];\,N_{\fk{m}}, L_{\fk{m}};\,\left\vert\underline{\alpha}_{ikj}\right\vert,\left\vert\underline{\alpha}_{ikj}\right\vert')
\]
is a Morita context. 
\end{thm}

\begin{proof}
Combine Lemmas \ref{lem-column-excision-bimodule-structure}, 
\ref{lem-row-excision-bimodule-structure}, 
\ref{lem-column-row-ligation-gives-bimodule-homomorphism-with-surjectivity-criteria}, 
\ref{lem-row-column-ligation-gives-bimodule-homomorphism-with-surjectivity-criteria} and \ref{lem-mixed-associativity-conditions-between-excised-matrix-bimodules}. 
\end{proof}


\begin{cor}
\label{cor-main-theorem-about-contexts-restricts-to-equivalences}
The context given by \Cref{thm-induced-contexts-for-generalised-matrix-rings} defines a Morita equivalence $[A_{i};\,M_{ij}]\sim [A_{i}';\,M_{ij}']$ assuming additionally that: the rings $A(i),R,S$ all have local units; that the bimodules $M_{ij},N,L$ are all unital; and that the maps  $\theta,\zeta$ are surjective. 
\end{cor}

\begin{proof}
Apply Lemmas  \ref{lem-column-row-ligation-gives-bimodule-homomorphism-with-surjectivity-criteria} and 
\ref{lem-row-column-ligation-gives-bimodule-homomorphism-with-surjectivity-criteria} and Theorems \ref{thm-recalled-from-Anh-Marki} and \ref{thm-induced-contexts-for-generalised-matrix-rings}. 
\end{proof}

\begin{proof}[Proof of Proposition \ref{prop-induced-contexts-for-calssical-matrix-rings}]
Specify \Cref{cor-main-theorem-about-contexts-restricts-to-equivalences} to $t=n=2$, $(1-e)R(1-e)=A_{1}=B_{1}$, $A_{2}=eRe$ and $B_{2}=S$.
\end{proof}

We end \S\ref{sec-excision-and-ligation} with some remarks. 
Firstly, the proof of our main result can probably be found using an inductive argument, however our direct approach is still rather elementary (and perhaps illustrative). 
Secondly, our results likely hold even after extending the class of rings we are considering, rings with local units, to the class of all idempotent rings. 
Note that Morita theory for idempotent rings was developed by Garc\'{\i}a and Sim\'{o}n \cite{Garcia-Simon-Morita-equivalence-idempotent-rings}.

\section{Examples and applications}
\label{sec-examples}

The author believes the following examples present interesting applications for the Morita theory of corner replacement. 

\subsection{Triangular matrix rings} Let $R$, $S$, and $T$ be unital rings and let $M$ be an $R$-$T$-bimodule. 
Hence one can consider the Morita context $(R,T; M, 0; 0, 0)$. 
Suppose that the rings $R$ and $S$ are Morita equivalent, and hence there is a Morita context $(R,S; N,L;\zeta,\theta)$ where $\zeta$ and $\theta$ are surjective bimodule homomorphisms. 
Composing $(R,S; N,L;\zeta,\theta)$ and $(R,T; M, 0; 0, 0)$ defines another Morita context $(S,T; L\otimes_{R}M,0;0,0)$ by \Cref{lem-extension-is-a-Morita-context-associativity}. 

In fact, by \Cref{cor-main-theorem-about-contexts-restricts-to-equivalences} the generalised matrix rings 
\[
\begin{array}{cc}
\begin{bmatrix}
R & M \\
0 & T
\end{bmatrix},
     &  
     \begin{bmatrix}
S & L\otimes M \\
0 & T
\end{bmatrix}
\end{array}
\]
given respectively by the contexts $(R,T; M, 0; 0, 0)$ and $(S,T; L\otimes_{R}M,0;0,0)$ are Morita equivalent, and this is precisely a result of Chen and Zheng \cite[Corollary 4.7]{ChenZheng-comma} concerning Morita equivalent upper-triangular matrix rings.  
As an example, Morita equivalent finite-dimensional algebras have Morita equivalent trivial extensions; see \cite[Corollary 4.8]{ChenZheng-comma}. 
Such Morita equivalences were used in work on the symmetry of the finitistic dimension; see for example work of Krause \cite[Lemma 5]{Krause-symmetr-finitistic}.

\subsection{Tensor rings of finite prospecies}
\label{subsec-Presenting-prospecies-by-quivers-and-relations}
We follow work of K\"{u}lshammer. 
Let $n\geq 1$ be an integer and let $K$ be an algebraically closed field. 
Consider a \emph{finite} \emph{pro}-\emph{species} over $K$. 
So, for each $i=1,\dots, n$ we let $A_{i}$ be a finite-dimensional $K$-algebra, and for each $i,j$ let 
$N_{ij}$ be an $A_{j}$-$A_{i}$-bimodule that is finitely-generated and projective both as a left $A_{j}$-module and also as a right $A_{i}$-module. 
Now we let 
\[
\begin{array}{cc}
M_{ij}=\bigoplus_{\ell\geq 1}\bigoplus_{d(1),\dots,d(\ell)=1}^{n}N_{id(1)}\otimes_{A_{d(1)}}\dots \otimes_{A_{d(\ell)}}N_{d(\ell)j}, 
&
M_{ji}=0.
\end{array}
\]
Thus $(A_{i}; M_{ij};\varphi_{ikj})$  is a generalised Morita context where $\varphi_{ikj}\colon M_{ik}\otimes_{A_{k}} M_{kj}\to M_{ij}$ is the canonical inclusion in the  direct sum. 
Now suppose the finite pro-species fixed above is \emph{dualisable} (in the sense of \cite[Definition 4.4]{Kulshammer-pro-species}), meaning that $\Hom_{\rMod{A_{i}}}(M_{ij},A_{i})\cong \Hom_{\lMod{A_{j}}}(M_{ij},A_{j})$ as $A_{j}$-$A_{i}$-bimodules for each $i,j$. 
Here the generalised matrix ring $[A_{i};M_{ij}]$ is also a finite-dimensional algebra, and is hence Morita equivalent to an admissible quotient $KQ/J$ of a path algebra $KQ$ of a quiver $Q$ by Gabriel's theorem. 
We can compute $Q$ and $J$ by understanding each $A_{i}$ and each $N_{ij}$ as follows. 

Firstly, again by Gabriel's theorem, for each $i$ there is a quiver $Q_{i} $ and an admissible ideal $J_{i}$ of $KQ_{i}$ such that  $A_{i}$ is Morita equivalent to  $B_{i}=KQ_{i}/J_{i}$. 
Let us write $(A_{i},B_{i};N_{i},L_{i};\zeta_{i},\theta_{i})$ for the context defining this  Morita equivalence. 
Secondly, since they define a Morita equivalence, the bimodules $N_{i}$ and $L_{i}$ are both finitely generated and projective on either side;  see for example \cite{Bass-algebraic-K-theory}. 
Thirdly, as noted just before \cite[Definition 2.2]{Kulshammer-pro-species}, this means that for each $i,j=1,\dots,n$ we have that  $L_{i}\otimes _{A_{i}} M_{ij}\otimes_{A_{j}}N_{j}$ is also finitely generated and projective both as a left $B_{i}$-module and a right $B_{j}$-module. 

Hence the generalised Morita context 
$(A_{i}; M_{ij};\varphi_{ikj})$ can be composed with each of the Morita contexts $(A_{i},B_{i};N_{i},L_{i};\zeta_{i},\theta_{i})$ using row-column ligation and column-row ligation to form a new generalised Morita context $(B_{i}; M'_{ij};\varphi'_{ikj})$ that must necessarily define a pro-species. 
Furthermore by  \Cref{cor-main-theorem-about-contexts-restricts-to-equivalences} the generalised matrix rings $[A_{i};M_{ij}]$ and $[B_{i};M'_{ij}]$ are Morita equivalent. 
Finally, one obtains a quiver $Q$ and an admissible ideal $J$ of $KQ$ using the quivers $Q_{i}$ and the ideals $J_{i}$ by means of \cite[Proposition 7.1]{Kulshammer-pro-species}.

\subsection{Semilinear clannish algebras and orbifold triangulations}
\label{subsec-Semilinear-clannish-algebras-and-orbifold-triangulations}



Let $F$ be a field containing primitive $4^{\text{th}}$ root of unity $\zeta\in F$. 
Let $E/F$ be a degree-$4$ cyclic Galois extension, say where $\mathrm{Gal}_{F}(E)=\{\myid,\rho,\rho^{2},\rho^{3}\}$. 
Let $L=\{\ell\in E\mid \rho^{2}(\ell)=\ell\}=F(u)$, $u=v^{2}$ for $0\neq v\in E$ with $u^{2}=-1$,  $\rho(v)=\zeta v$. 
Fix $\theta\in \mathrm{Gal}(E/F)$  and write $L^{\theta}$ for the $L$-bimodule whose right action is twisted by $\theta$.
 
In work of Geuenich and Labardini-Fragoso \cite{Geuenich-labardini-fragoso-surfaces} these authors associate, to the \emph{self}-\emph{folded} triangle connecting orbifold points on the left, the $F$-species mnemotechnically described in the middle, which itself is associated to the weighted quiver of the right. 
\[
\begin{array}{ccc}
\begin{array}{c}
\includegraphics[scale=.13]{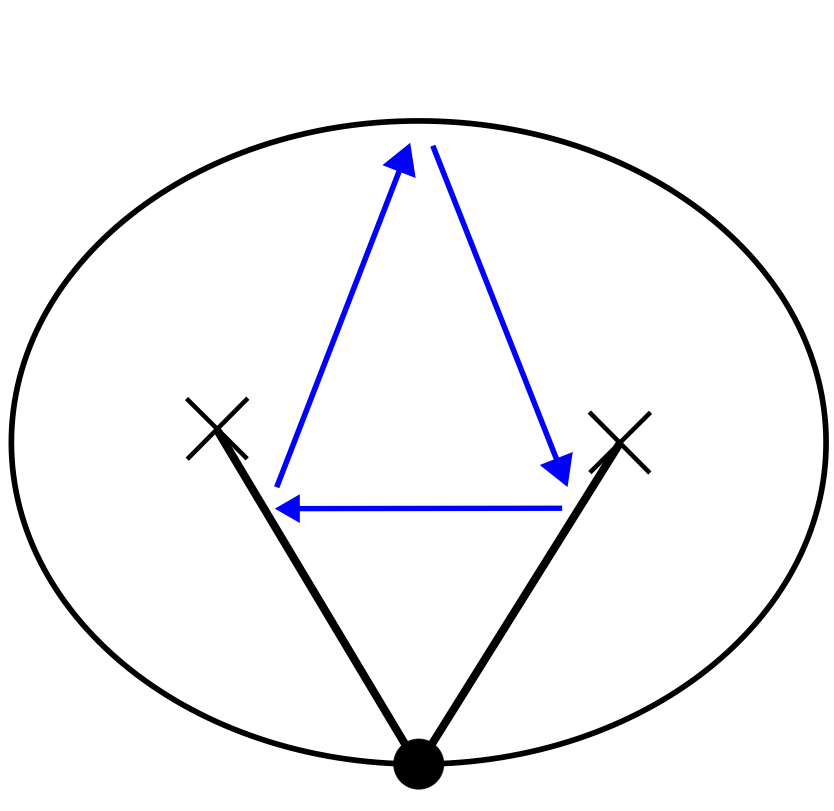}
\end{array}
       &
       \begin{array}{c}
\begin{tikzcd}
     & L  \arrow[dr, "F\otimes_{F} L"] & \\
 E \arrow[ur, "L^{\theta}\otimes_{L} E"] & & F \arrow[ll, "E\otimes_{F} F"]
\end{tikzcd}

\\
\end{array}
&
\left( \begin{array}{cc}
\tiny{
\begin{tikzcd}[column sep=0.3cm]
& 1 \arrow[dr, "\gamma"] & &\\
2 \arrow[ur, "\alpha"] & & 3\arrow[ll, "\beta"] &
\end{tikzcd}
\begin{tikzcd}[column sep=0.05cm]
& 2=[L:F]  & \\
 4=[E:F]  & & 1=[F:F]
\end{tikzcd}}
\end{array}\right)
\end{array}
\]
Cyclic derivatives of the potential $\alpha\beta\gamma$ give the  relations $\frac{1}{2} (u\beta\gamma \pm \beta\gamma u)$, $\gamma\alpha$ and $\alpha\beta$ (the sign depends on $\theta$). 
The tensor ring $\Gamma$ of this species with potential has presentation
\[
\begin{bmatrix}
L & L^{\theta}\otimes_{L} E & 0 \\
E\otimes_{F}L/\langle w\rangle & E & E\otimes_{F} F \\
F\otimes_{F} L & 0 & F
\end{bmatrix}\sim 
\begin{bmatrix}
L & L^{\theta}\otimes_{L} \frac{L[x]}{\langle x^{2} - u\rangle} & 0 \\
\frac{L[x]}{\langle x^{2} - u\rangle} \otimes_{F}L/\langle w\rangle & \frac{L[x]}{\langle x^{2} - u\rangle} & \frac{L[x]}{\langle x^{2} - u\rangle}\otimes_{F} \frac{L[y;\tau]}{\langle y^{2}-1\rangle} \\
\frac{L[y;\tau]}{\langle y^{2}-1\rangle}\otimes_{F} L & 0 & \frac{L[y;\tau]}{\langle y^{2}-1\rangle}.
\end{bmatrix}
\]
where $w=u\otimes 1 \pm 1\otimes  u$ and where the Morita equivalence is found by applying Proposition \ref{prop-induced-contexts-for-calssical-matrix-rings} twice using the ring isomorphism $E\cong L[x]/\langle x^{2} - u\rangle$ and the Morita equivalence  $F\sim M_{2}(F)\cong L[y;\rho^{2}]/\langle y^{2}-1\rangle$;  see for example \cite[Theorem 1.3.16]{Jacobsen-division-rings}.  

Hence $\Gamma$ is Morita equivalent to 
a \emph{semilinear clannish algebra} in the sense of joint work of the author with Crawley-Boevey \cite{BTCB-semilinear-clannish}. 
In joint work of the author with Labardini-Fragoso \cite{BTLabardini-Fragoso2023semilinear} 
such equivalences have been treated using an alternative and  explicit approach. 
The equivalence above is found in \cite[Table 5]{BTLabardini-Fragoso2023semilinear}. 

\subsection{Rings with enough idempotents}
\label{subsec-rings-with-enough-idempotents}

Let $A$ be a ring with a complete  set  $E=\{e_{i}\mid i\in I\}$  of {orthogonal} idempotents $e_{i}=e_{i}^{2}\in A$, so that 
 $e_{i}e_{j}=0$ for $i\neq j$ and $A=\bigoplus_{i\in I}e_{i}A=\bigoplus_{i\in I}Ae_{i}$. 
Then $A$ has local units.  
Partition $I=I_{-}\sqcup I_{+}$ to give a partition $E=E_{-}\sqcup E_{+}$ and hence a Morita context $(A_{-},A_{+};\,M_{-+},M_{+-};\,\varphi_{-},\varphi_{+})$ where 
\[
\begin{array}{ccc}
A_{\pm}=
\bigoplus_{e,e'\in E_{\pm}} e'Ae,
&
M_{\pm\mp}=\bigoplus_{e'\in E_{\pm},e\in E_{\mp}} e'Ae,
&
\varphi_{\pm}(a\otimes b)=ab.
\end{array}
\]
By construction we have $A\cong \begin{bsmallmatrix}
    A_{+} & M_{+-}\\
    M_{-+} & A_{-}
\end{bsmallmatrix}$ as rings. 
Note also each bimodule $M_{\pm\mp}$ is unital.

Suppose $(A_{\pm},B_{\pm};\,N_{\pm},L_{\pm};\,\theta_{\pm},\zeta_{\pm})$ are Morita contexts where $N_{\pm},L_{\pm}$ are unital. 
Note  $L_{\pm}e'\otimes_{A_{\pm}}eAe'' =0$ and $e''Ae\otimes_{A_{\pm}} e 'N_{\pm} =0$ for any distinct $e,e'\in E_{\pm}$. 
Hence, as tensor products commute over direct sums (see for example \cite[\S 12.4]{Wisbauer-foundations}),  we have 
\[
\begin{array}{cc}
L_{\pm}\otimes_{A_{\pm}}M_{\pm\mp}=\bigoplus_{e\in E_{\pm},e'\in E_{\mp}}L_{\pm}e\otimes_{A_{\pm}}eAe'
,
&
M_{\pm\mp}\otimes_{A_{\mp}}N_{\mp}=\bigoplus_{e\in E_{\pm},e'\in E_{\mp}}eAe'\otimes_{A_{\mp}}e'N_{\mp}.
\end{array}
\] 
Hence we obtain a more explicit description of the column excision and row excision from \S\ref{sec-excision}. 
Note that the column-row ligation and row-column-ligation from \S\ref{sec-ligation} are defined using the multiplication in $A$. 
As in \Cref{cor-main-theorem-about-contexts-restricts-to-equivalences}, when the maps $\theta_{\pm},\zeta_{\pm}$ are all surjective  the classical Morita context from \Cref{thm-induced-contexts-for-generalised-matrix-rings} gives a Morita equivalence \[
A\sim \begin{bmatrix}
    B_{+} & \bigoplus_{e_{\pm}\in E_{\pm}} L_{+}e_{+}\otimes e_{+}Ae_{-}\otimes e_{-}N_{-} \\
    \bigoplus_{e_{\pm}\in E_{\pm}} L_{-}e_{-}\otimes e_{-}Ae_{+}\otimes e_{+}N_{+} & B_{-} 
\end{bmatrix}.
\] 

\subsection{Functor categories}

We recall the notion of a  \emph{functor ring} following Fuller \cite{Fuller-functor-rings}, using a result due to Gabriel \cite{Gabriel-des-categories-abeliennes-1962}. We then recall a result due to Dung and Garc\'{\i}a \cite{Dung-Garcia-additive-category} relating the first  idea with rings with enough idempotent. 
Such ideas have been combined before; see for example work of Garc\'{\i}a, G\'{o}mez S\'{a}nchez, and Mart\'{\i}nez Hern\'{a}ndez 
\cite[\S1]{Garcia-Gomez-Sachex-Martinez-loc-fin-pres-cats-and-functor-rings}. 
Functor rings were also used by Harada \cite{Harada-Perfect-cats-I,Harada-Perfect-cats-II} where they studied \emph{perfect} categories. 
Yamagata  \cite{Yamagata-Morita-duality-1979} also used this idea to give a Morita duality theory for certain functor categories.


 Let $\mathcal{C}$ be a skeletally small additive category and let $\mathcal{A}$ be a chosen skeleton.  
 The \emph{functor ring} of $\mathcal{C}$ (with respect to $\mathcal{A}$) is the direct sum of additive groups
\[
\begin{array}{c}
\Func_{\mathcal{A}}(\mathcal{C})=\bigoplus_{X,Y\in\mathcal{A}}\Hom_{\mathcal{C}}(X,Y).
\end{array}
\]
The multiplication map $\Func_{\mathcal{A}}\times \Func_{\mathcal{A}}\to \Func_{\mathcal{A}}$ is defined as the universal morphism given by the diagram of additive maps $ \Hom_{\mathcal{C}}(Y,Z)\times \Hom_{\mathcal{C}}(X,Y)\to \Hom_{\mathcal{C}}(X,Z)$  given by the composition in $\mathcal{C}$. 
Writing $\myid_{X}$ for the identity on an object $X$, it is straightforward to see that  $\{\myid_{X}\mid X\in \mathcal{A}\}$ defines a complete set $E_{\mathcal{A}}$ of orthogonal idempotents in $\Func_{\mathcal{A}}(\mathcal{C})$. 

By \cite[p. 347, Proposition 2]{Gabriel-des-categories-abeliennes-1962}, if $A=\Func_{\mathcal{A}}(\mathcal{C})$ then there is an equivalence between the category $\lMod{A}$ of unital left $A$-modules and the category $(\mathcal{C},\mathbf{Ab})$ of functors $\mathcal{C}\to \mathbf{Ab}$ where $\mathbf{Ab}$ is the category of abelian groups. 
Following Crawley-Boevey \cite[p. 1642]{Crawley-Boevey-Locally-finitely-presented-additive-categories}, if $\mathcal{C}$  has direct limits then  an object $Z$ is \emph{finitely presented} if $\mathcal{C}(Z,-)$ commutes with direct limits, and  $\mathcal{C}$ is \emph{locally finitely presented} if the finitely presented objects form a skeletally small 
subcategory where every object
in $\mathcal{C}$ is a direct limit of finitely presented ones. 
By  \cite[Theorem 1.1]{Dung-Garcia-additive-category} there a one-to-one correspondence between rings with enough idempotents up to Morita equivalence and locally finitely presented additive categories up to equivalence. 

Hence one can use the discussion so far to reinterpret \Cref{cor-main-theorem-about-contexts-restricts-to-equivalences} in terms of equivalence classes of locally finitely presentable additive categories. 
The author is interested in formalising this interpretation in future work. 

\begin{acknowledgements}
The author is grateful for  support by the Danish National Research Foundation (DNRF156); 
 the Independent Research Fund Denmark (1026-00050B); 
 and the Aarhus University Research Foundation (AUFF-F-2020-7-16). 
The author is also grateful to 
Benjamin Briggs, Peter J{\o}rgensen and Toby Stafford for helpful discussions.
\end{acknowledgements}
\bibliographystyle{bibstyle}
\bibliography{biblio}
\end{document}